\documentclass[11pt]{article}
\usepackage[margin=1in]{geometry}
\usepackage{amsmath,amssymb,amsthm,mathtools}
\usepackage{booktabs}
\usepackage[expansion=false]{microtype}
\usepackage[hidelinks]{hyperref}
\usepackage{enumitem}
\usepackage{longtable}
\usepackage{array}
\usepackage{xcolor}

\numberwithin{equation}{section}
\newtheorem{theorem}{Theorem}
\newtheorem{proposition}[theorem]{Proposition}
\newtheorem{lemma}[theorem]{Lemma}
\newtheorem{corollary}[theorem]{Corollary}
\theoremstyle{definition}
\newtheorem{definition}[theorem]{Definition}
\newtheorem{example}[theorem]{Example}
\theoremstyle{remark}
\newtheorem{remark}[theorem]{Remark}

\newcommand{\Z}{\mathbb Z}
\newcommand{\Q}{\mathbb Q}

\newcommand{\fall}[2]{(#1)_{#2}}
\newcommand{\OEIS}[1]{\href{https://oeis.org/#1}{\texttt{#1}}}
\title{Universal Shift Congruences for Weighted Permutation Sequences}
\ifdefined\AnonymousSubmission
  \author{Anonymous manuscript}
  \hypersetup{pdftitle={Universal Shift Congruences for Weighted Permutation Sequences},pdfauthor={Anonymous manuscript}}
\else
  \author{Ahaan Kallat\\Northeastern University\\\texttt{kallat.a@northeastern.edu}}
  \hypersetup{pdftitle={Universal Shift Congruences for Weighted Permutation Sequences},pdfauthor={Ahaan Kallat}}
\fi
\date{October 2026}

\begin{document}
\maketitle

\begin{abstract}
The involution numbers $1,1,2,4,10,26,\ldots$ repeat modulo $3$, $5$ and $7$.  We explain this, and many similar patterns, by one rule.  Many sequences count permutations in which every cycle of length $m$ carries a weight $w_m$ (involutions, derangements, ``sets of lists'', and many others in the OEIS).  Writing $A(n)$ for such a count, the rule is $A(n+k)\equiv A(n)A(k)$ modulo $\operatorname{lcm}(n,k)$.  It generalizes our earlier proof of the case modulo $k$ for one such sequence, but now it holds for all weights at once, as a coefficientwise congruence between universal polynomials, and we find the largest modulus that works for the universal polynomials, $nk/\operatorname{rad}(\gcd(n,k))$, and show that nothing larger does.  We then show that passing from $k$ to $pk$ behaves like raising to the $p$-th power modulo a high power of $p$, with a sharp exponent for the polynomials, which gives $A(p^r)\equiv(w_1+(p-1)w_p)^{p^{r-1}}$ modulo $p^r$ for odd $p$, and explains why the prime $2$ is special.  A family of sequences defined by Lagrange inversion fits the same framework, which proves Bala's conjecture for the whole family (the single sequence \OEIS{A278070} had been proved earlier by an automated system) and the divisibility observed in the entries \OEIS{A251568} and \OEIS{A251569}.  As further applications we prove OEIS conjectures of Richman, Hanna, van Hoeij and Bala (\OEIS{A227937}), and the theorems apply to 101 OEIS sequences, whose congruences are recorded in a supplementary catalogue.  Higher-order versions of the content formula are treated in a separate paper.
\end{abstract}


\section{Introduction}\label{sec:introduction}

The \emph{involution numbers} $A(n)$ count the ways to arrange $n$ people either alone or in pairs.\footnote{This is the OEIS sequence \OEIS{A000085}.}  Reduced modulo a small number they are surprisingly regular.
\begin{center}
\renewcommand{\arraystretch}{1.0}
\begin{tabular}{r|cccccccccc}
$n$ & 0 & 1 & 2 & 3 & 4 & 5 & 6 & 7 & 8 & 9\\\hline
$A(n)$ & 1 & 1 & 2 & 4 & 10 & 26 & 76 & 232 & 764 & 2620\\
$A(n)\bmod 3$ & 1 & 1 & 2 & 1 & 1 & 2 & 1 & 1 & 2 & 1\\
$A(n)\bmod 5$ & 1 & 1 & 2 & 4 & 0 & 1 & 1 & 2 & 4 & 0\\
$A(n)\bmod 7$ & 1 & 1 & 2 & 4 & 3 & 5 & 6 & 1 & 1 & 2
\end{tabular}
\end{center}
Modulo $3$ the block $1,1,2$ repeats forever, and modulo $5$ and $7$ blocks of length $5$ and $7$ repeat.  Indeed $A(3)=4$, $A(5)=26$ and $A(7)=232$ are each one more than a multiple of the modulus, which suggests the rule $A(n+k)\equiv A(n)A(k)\pmod k$.

The rule holds for a large class of sequences.  Fix a sequence of integers $w=(w_1,w_2,\ldots)$, give each permutation of $\{1,\ldots,n\}$ the weight obtained by letting every cycle of length $m$ contribute a factor $w_m$, and let $A_w(n)$ be the sum of the weights of all $n!$ permutations ($A_w(0)\coloneqq1$).  Involutions are given by $w_1=w_2=1$ and $w_m=0$ otherwise, while derangements\footnote{The OEIS sequence \OEIS{A000166}.} are given by $w_1=0$, $w_m=1$ for $m\ge2$.  We also let the weights be unknowns, and write $C_n(X_1,\ldots,X_n)$ for $A_w(n)$ with $w_j$ replaced by an indeterminate $X_j$, so that for example $C_1=X_1$ and $C_2=X_1^2+X_2$.\footnote{In other words, $C_n$ is $n!$ times the cycle index polynomial of the symmetric group $S_n$, and Proposition~\ref{prop:cycle-index} gives its coefficients.}  Proving a congruence once for the polynomials $C_n$ proves it for every weight sequence.

\begin{theorem}[Shift congruence]\label{thm:universal}
For every $n,k\ge1$,
\[
 C_{n+k}\equiv C_nC_k
 \pmod{\operatorname{lcm}(n,k)\Z[X_1,X_2,\ldots]}.
\]
Consequently, for every integer weight sequence $w$ and all $n,k\ge1$,
\[
 A_w(n+k)\equiv A_w(n)A_w(k)\pmod{\operatorname{lcm}(n,k)},
\]
and in particular $A_w(n+k)\equiv A_w(n)A_w(k)\pmod k$ for all $n\ge0$, $k\ge1$.
\end{theorem}

\begin{remark}
The first line is what the proofs give, and the others are special cases.  Taking $w_1=w_2=1$ and all other weights $0$ recovers the periods $3,5,7$ in the table above, because the theorem says $A(n+k)\equiv A(n)$ whenever $A(k)\equiv1\pmod k$.  Since the congruence is coefficientwise, it survives every substitution for the variables, including substitutions by polynomials (Remark~\ref{rem:substitutions}).  For a fixed modulus $k$, the shift congruence is equivalent to each residue-class subsequence $A_w(r+qk)$ being a geometric progression with ratio $A_w(k)$ modulo $k$ (Remark~\ref{thm:fixed-modulus-classification}).
\end{remark}

A short calculation (Proposition~\ref{prop:cycle-index}) shows that the numbers $A_w(n)$ are the coefficients of an exponential generating function (e.g.f.):
\[
 \sum_{n\ge0}A_w(n)\frac{x^n}{n!}=\exp\Bigl(\sum_{m\ge1}w_m\frac{x^m}{m}\Bigr).
\]
Conversely, every power series with constant term $1$ equals $\exp\bigl(\sum_{m\ge1}u_mx^m/m\bigr)$ for unique rational numbers $u_m$, which are its \emph{logarithmic coordinates} (Section~\ref{sec:coordinates-intro}), and the series coming from integer weights are exactly those whose logarithmic coordinates are integers.  Many entries of the On-Line Encyclopedia of Integer Sequences (OEIS)~\cite{OEIS} have this form.

\begin{example}[Involutions modulo $4$]\label{ex:mod4}
For $k=4$ we have $A(4)=10\equiv2\pmod4$, so the theorem predicts only $A(n+4)\equiv2A(n)\pmod 4$.  Indeed modulo $4$ the involution numbers are $1,1,2,0,2,2,0,0,0,\ldots$, and the factor $2$ keeps multiplying and eventually kills everything, so there is no period.  This shows that the factor $A_w(k)$ cannot be dropped from the theorem, and that a sequence satisfying the shift congruence need not be periodic.
\end{example}

\begin{example}[Integrality of the logarithmic weights]\label{ex:integrality}
The series $1+x^2/2$ has integer exponential coefficients $a_0=1$, $a_1=0$, $a_2=1$ and $a_n=0$ for $n\ge3$, but $a_4-a_2^2=-1$ is odd, so $a_{n+k}\equiv a_na_k\pmod k$ fails at $n=k=2$.  Here $\log(1+x^2/2)=x^2/2-x^4/8+\cdots$, so $u_4=-\tfrac12$ is not an integer.  Thus integrality of the numbers $A(n)$ alone is not enough, and integrality of the logarithmic coordinates is the right hypothesis.
\end{example}

\begin{remark}
For a prime-power modulus $p^e$ the proofs apply verbatim to weights in $\Z_{(p)}$, since the polynomial congruences hold in $\Z[X_1,X_2,\ldots]$.  The same applies to any modulus $k$ when the weights $w_m$ are rational numbers whose denominators are coprime to $k$ (that is, $w_m\in\Z_{(k)}$, the rationals with denominator prime to $k$), and in that case $A_w(n)\in\Z_{(k)}$ and $A_w(n+k)\equiv A_w(n)A_w(k)\pmod{k\Z_{(k)}}$.  When the numbers $A_w(n)$ are integers this is a congruence in $\Z$, because an integer lying in $k\Z_{(k)}$ is divisible by $k$.  There are higher-order analogues of the content formula, for products of several blocks.
\end{remark}

Here is a heuristic for the rule.  Split the $n+k$ points into a set $P$ of $n$ points and a set $Q$ of $k$ points, and let the cyclic group $\Z/k$ act on the permutations by rotating the points of $Q$ along a fixed $k$-cycle.  This action preserves the cycle type of a permutation, and hence its weight.  A permutation with a cycle meeting both $P$ and $Q$ is moved by every nontrivial rotation, so these permutations fall into orbits of exactly $k$ elements, all of the same weight, and their total weight is a multiple of $k$.

What is left are the permutations in which no cycle meets both sets.  These are the pairs consisting of a permutation of $P$ and a permutation of $Q$, their weights multiply, and so they contribute exactly $A(n)A(k)$.  Thus $A(n+k)\equiv A(n)A(k)\pmod k$.  This argument is made rigorous in Section~\ref{sec:universal} (Remark~\ref{rem:group-action}), and it is the argument of~\cite[Proposition~2.2]{KallatBala}.

\subsection{Notation}\label{sec:notation}

We write $[x^n]F$ for the coefficient of $x^n$ in a power series $F$, and $\fall{n}{r}\coloneqq n(n-1)\cdots(n-r+1)$ for the falling factorial.  For a prime $p$, $v_p(m)$ is the exponent of $p$ in $m$, and $\operatorname{rad}(m)$ is the product of the distinct primes dividing $m$.  The \emph{content} of a polynomial with integer coefficients is the greatest common divisor of its coefficients.  For a partition $\lambda$ with $m_j$ parts equal to $j$, $X_\lambda\coloneqq\prod_jX_j^{m_j}$, and $[X_\lambda]P$ is the coefficient of $X_\lambda$ in a polynomial $P$.  For polynomials, $P\equiv Q\pmod m$ means that every coefficient of $P-Q$ is divisible by $m$.\footnote{This implies the same congruence after substituting any integers for the variables, but not conversely, since $X^2-X$ is even at every integer and yet has odd coefficients.}

\subsection{The best modulus}

The modulus $\operatorname{lcm}(n,k)$ in Theorem~\ref{thm:universal} can be improved slightly, and the improvement is best possible.

\begin{theorem}[The best modulus, see Theorem~\ref{thm:optimal-divisor}]\label{thm:intro-best-modulus}
For $n,k\ge1$, the content of $C_{n+k}-C_nC_k$ is exactly $nk/\operatorname{rad}(\gcd(n,k))$.  Consequently $A_w(n+k)\equiv A_w(n)A_w(k)$ modulo $nk/\operatorname{rad}(\gcd(n,k))$ for all integer weights, and no larger modulus works for the polynomials, that is, for weights in an arbitrary commutative ring.
\end{theorem}

\begin{example}[The modulus $\operatorname{lcm}(n,k)$ is not always optimal]\label{ex:lcm}
For $n=k=2$ one finds $C_4-C_2^2=4X_1^2X_2+2X_2^2+8X_1X_3+6X_4$, whose content is $2=\operatorname{lcm}(2,2)$, so here $\operatorname{lcm}(n,k)$ cannot be improved.  The number $nk/\operatorname{rad}(\gcd(n,k))$ equals $\operatorname{lcm}(n,k)$ when $\gcd(n,k)$ has no repeated prime factor (for instance when $n$ and $k$ are coprime), and is larger otherwise, as for $n=k=4$, where it is $8$ and not $4$.  Thus the true modulus depends on the repeated prime factors of $\gcd(n,k)$, and this is what the lifting theorem below exploits.
\end{example}

\begin{remark}
For integer weights alone a slightly larger modulus can occasionally hold, since a polynomial can vanish modulo $m$ at every integer point without having all coefficients divisible by $m$ (Remark~\ref{rem:poly-vs-integer}).  The right reading of the theorem is that $nk/\operatorname{rad}(\gcd(n,k))$ is the exact modulus for the universal polynomial.
\end{remark}

\subsection{Lifting at every prime}

The shift congruence gives $A(pk)\equiv A(k)^p\pmod k$.  When $p$ divides $k$ this holds to a much higher power of $p$, much as $x\mapsto x^p$ is the Frobenius map in characteristic $p$.

\begin{theorem}[Lifting at every prime, see Theorem~\ref{thm:strong-all-primes}]\label{thm:intro-lift}
Let $p$ be a prime, let $p\mid k$, and put $s\coloneqq v_p(k)$.  Then every coefficient of $C_{pk}-C_k^{\,p}$ is divisible by $p^{2s}$ if $p$ is odd, and by $p^{2s-1}$ if $p=2$.  In particular $A_w(pk)\equiv A_w(k)^p$ modulo these powers of $p$, for all integer weights.  As statements about the polynomials, these powers are attained for every $k$ and every prime (Remark~\ref{rem:lift-sharp}).
\end{theorem}

\begin{example}[Involutions]\label{ex:lift}
For $p=k=3$ we have $A(9)-A(3)^3=2620-64=2556=3^2\cdot284$, divisible by $9$, whereas the shift congruence alone gives only divisibility by $3$.  For $p=2$ and $k=4$ (so $s=2$) we have $A(8)-A(4)^2=764-100=664=2^3\cdot83$, divisible by $2^{2s-1}=8$, whereas the shift congruence alone gives only divisibility by $4$.  In both cases the lift gains a factor of $p$ over the shift congruence.
\end{example}

Iterating the lift computes $A_w(p^r)$ modulo $p^r$ in closed form, since modulo $p^r$ the number $A_w(p^r)$ is the $p^{r-1}$-th power of the single number $w_1+(p-1)w_p$.

\begin{theorem}[Closed formulas modulo prime powers, see Section~\ref{sec:odd}]\label{thm:intro-closed}
Let $w$ be an integer weight sequence and $r\ge1$.
\begin{enumerate}[label=(\alph*)]
\item If $p$ is an odd prime, then $A_w(p^r)\equiv\bigl(w_1+(p-1)w_p\bigr)^{p^{r-1}}\pmod{p^r}$.
\item For $p=2$ the same formula, $A_w(2^r)\equiv(w_1+w_2)^{2^{r-1}}\pmod{2^r}$, holds for $r=1$ and for every $r\ge4$.  The prime $2$ is exceptional at two levels only, $r=2$ and $r=3$, where the weight $w_4$ also enters.
\item Consequently, for $r\ge4$ the residue of $A_w(2^r)$ modulo $2^r$ is $1$ if $w_1+w_2$ is odd and $0$ if it is even.
\end{enumerate}
\end{theorem}

\begin{remark}
For involutions, (a) gives $A(3^r)\equiv1\pmod{3^r}$.  A further consequence (Corollary~\ref{cor:eventual-divisibility}) is the cleanest criterion of the paper, namely that $A_w(n)$ is divisible by $p^e$ for all large $n$ if and only if $p$ divides $w_1+(p-1)w_p$.
\end{remark}

\subsection{Lagrange-type families}

A second family of sequences does not obviously count weighted permutations, yet turns out to.  It is defined by Lagrange inversion, by taking the $n$-th power of a series.

\begin{theorem}[Lagrange-type families, see Theorem~\ref{thm:lagrange}]\label{thm:intro-lagrange}
Let $F,G\in\Z[[x]]$ be power series with integer coefficients, and put $b(n)\coloneqq n![x^n]\,e^{xG(x)}F(x)^n$.  Then $b(n)=A_w(n)$ for an integer weight sequence $w$, so all the results above apply.  Moreover $b(k)\equiv(F(0)G(0))^k\pmod k$.  In particular, if $F(0)=G(0)=1$ then $b(n+k)\equiv b(n)\pmod k$ for all $n\ge0$ and $k\ge1$.
\end{theorem}

\begin{remark}
From computer data Bala conjectured the last statement, and the theorem proves his conjecture for the whole family.  For instance, the OEIS sequence \OEIS{A278070} is the case $F=1/(1-x)$, $G=1$, and a proof of that single case by an automated system is recorded in its entry.
\end{remark}

\subsection{Applications to the OEIS}

The theorems above apply to 101 sequences from the OEIS, and the resulting congruences are collected in a supplementary catalogue, provided as an ancillary file with this paper.  In Section~\ref{sec:sequences} we prove four conjectures recorded in the OEIS, namely Richman's conjecture that $(-1)^ma(m)\equiv(-1)^na(n)\pmod q$ for $m\equiv n\pmod q$ with $q$ odd (\OEIS{A053871}), Hanna's conjecture that $p\mid a(n)$ for $n\ge p$ whenever $p$ is a prime of the form $m^2+1$ (\OEIS{A214724}), van Hoeij's conjecture that $k\mid a(n+k)-a(n)$ (\OEIS{A000459}), and Bala's conjecture that $n\mid a(n+k)$ for $n$ coprime to $6$ and $k\ge0$ (\OEIS{A227937}, partitions into blocks of size $2$ or $3$).

\subsection{Overview of the proofs}\label{sec:overview}

We now briefly discuss the methods of proof.  The difficulty is that each sequence $A_w(n)$ is a sum over $n!$ permutations with no closed form, so a congruence for one sequence typically uses special features of that sequence (its recurrence, or the arithmetic of its e.g.f.).  Roughly speaking, the key point is to prove congruences between the polynomials $C_n$ coefficient by coefficient, and \emph{not} to work with one sequence at a time.  This costs nothing for the shift congruence, but it makes the question of the exact modulus meaningful, and for that question the obvious routes stall.  For instance, a free group action on the permutations that mix two blocks gives divisibility by $k$ and nothing more, and iterating the shift congruence gains only one factor of $k$ when one needs two.

The first ingredient is the falling-factorial recurrence for $C_n$ (Section~\ref{sec:recurrence}), and comparing the recurrences for $C_{n+k+1}$ and $C_{n+1}$ modulo $k$ proves Theorem~\ref{thm:universal} by induction (Section~\ref{sec:universal}).  A second proof lets a rotation of order $k$ act freely on the permutations that mix the two parts (Remark~\ref{rem:group-action}).  The second ingredient is an orbit count for the exact modulus, in which a product of two symmetric groups acts on the mixing permutations by conjugation, an elementary $p$-adic inequality for factorials (Lemma~\ref{lem:mixed-support-factorial}) bounds the centralizers, and one explicit coefficient shows that nothing larger works (Section~\ref{sec:content}).  The third ingredient, and the main difficulty, is the lift.  Cutting a permutation of $p$ blocks of $k$ points into connected pieces, the pieces touching three or more blocks are divisible by $k^2$, so it remains to control the pieces touching exactly two blocks.  These are measured exactly by Theorem~\ref{thm:intro-best-modulus}, and the binomial coefficient $\binom p2$ explains the difference between odd primes and $2$ (Section~\ref{sec:lift}).  The fourth ingredient is an induction on $r$ using that $x\equiv y\pmod{p^r}$ implies $x^p\equiv y^p\pmod{p^{r+1}}$, which gives the closed formulas (Section~\ref{sec:odd}).  The last ingredient is Lagrange inversion, which makes the logarithm of the generating function of the Lagrange families explicit, with integer coefficients (Section~\ref{sec:lagrange}). 

\subsection{Relation to earlier work}\label{sec:relation}

Bala~\cite{Bala2017} proved in 2017 that $a(n+k)\equiv a(n)\pmod k$ when the e.g.f.\ is $F(x)\exp(xG(x))$ with $F,G\in\Z[[x]]$ and $G(0)=1$.  With $F(0)=1$, his class lies inside ours (Section~\ref{sec:a-k}).

The product congruence modulo $k$ was first proved in~\cite[Proposition~2.2]{KallatBala} (17 July 2026) for \OEIS{A028342}, by letting $\Z/k$ act on colored permutations.  The argument does not depend on the divisor structure and applies to arbitrary integer weights.  Bala~\cite{Bala2026} (31 July 2026) subsequently recorded the general-weight congruence, citing~\cite{KallatBala}, and gave a proof by induction on the recurrence.

The present paper proves the congruence in the universal polynomial form and with the modulus $\operatorname{lcm}(n,k)$ (Theorem~\ref{thm:universal}), by both the group-action argument and the recurrence.  What goes beyond~\cite{KallatBala} and~\cite{Bala2026} is the exact content $nk/\operatorname{rad}(\gcd(n,k))$ with its sharpness, the lift at every prime and its sharpness, the closed formulas, the Lagrange family, and the applications of Section~\ref{sec:sequences}.

Bellagh and Oulebsir~\cite{BellaghOulebsir2023} give a congruence from which the residue of $A_w(p^r)$ modulo $p^r$ follows for odd $p$ (Remark~\ref{rem:BO}).  Congruences for Bell numbers and polynomials~\cite{Carlitz1961,GertschRobert1996,Mihoubi2009,SunZagier2011,SunWuZhuang2013} concern one sequence at a time, whereas ours hold for all weights at once.  To our knowledge, the exact content formula, the lift at every $k$ and every prime together with its sharpness, the $2$-adic closed form, and the integrality of the logarithmic weights of the Lagrange family have not appeared before.  The provenance of the OEIS conjectures treated in this paper is discussed in Section~\ref{sec:provenance}.

\section{Weighted permutations, the cycle index and logarithmic coordinates}\label{sec:cycle}

In this section we recall the numbers $A_w(n)$ of the Introduction and prove the cycle-index formula, Proposition~\ref{prop:cycle-index}, on which every later section depends.  The \emph{universal} version $C_n$, in which the weights are unknowns, lets us prove a congruence for all weights at once.

Throughout, $w=(w_1,w_2,\ldots)$ is a sequence of integers, $A_w(x)=\exp\bigl(\sum_{m\ge1}w_mx^m/m\bigr)=\sum_{n\ge0}A_w(n)x^n/n!$, and $C_n(X_1,\ldots,X_n)$ is $A_w(n)$ with $w_j$ replaced by $X_j$ (these are the \emph{cycle index polynomials}, and $C_0=1$).  For a partition $\lambda=(1^{m_1}2^{m_2}\cdots)$ of $n$ (that is, $m_j$ cycles of length $j$) set
\[
 z_\lambda\coloneqq\prod_{j\ge1}j^{m_j}m_j!.
\]

\begin{proposition}[Cycle-index formula]\label{prop:cycle-index}
For every $n\ge0$,
\begin{equation}\label{eq:cycle-index}
 A_w(n)=\sum_{\lambda\vdash n}\frac{n!}{z_\lambda}\prod_jw_j^{m_j}.
\end{equation}
In particular, $A_w(n)\in\Z$, and $A_w(n)$ depends only on $w_1,\ldots,w_n$.
\end{proposition}

\begin{proof}
Expand each exponential factor:
\[
 A_w(x)=\prod_{j\ge1}\exp\!\left(w_j\frac{x^j}{j}\right)
 =\prod_{j\ge1}\sum_{m_j\ge0}\frac{w_j^{m_j}x^{jm_j}}{j^{m_j}m_j!}.
\]
Extracting the coefficient of $x^n$ and multiplying by $n!$ gives \eqref{eq:cycle-index}.  The quotient $n!/z_\lambda$ is the number of permutations of cycle type $\lambda$, hence is an integer.  So \eqref{eq:cycle-index} says that $A_w(n)=\sum_{\sigma\in S_n}\prod_{c}w_{|c|}$, the product being over the cycles $c$ of $\sigma$, which links the e.g.f.\ with the combinatorial definition in the Introduction.
\end{proof}

\begin{example}[The first universal polynomials]\label{ex:first-C}
Applying the cycle-index formula for $n\le4$ gives
\begin{align*}
 C_1&=X_1, & C_2&=X_1^2+X_2,\\
 C_3&=X_1^3+3X_1X_2+2X_3, & C_4&=X_1^4+6X_1^2X_2+3X_2^2+8X_1X_3+6X_4.
\end{align*}
Here the coefficient of $X_1^2X_2$ in $C_4$ is $6$ because there are $\binom42=6$ permutations of $\{1,2,3,4\}$ with one $2$-cycle and two fixed points.  For involutions ($X_1=X_2=1$, $X_j=0$ for $j\ge3$) we get $C_4=1+6+3=10$, matching the table in the Introduction.
\end{example}

\subsection{Logarithmic coordinates}\label{sec:coordinates-intro}

The proofs rest on the more general fact that \emph{every} power series with constant term $1$, not only those coming from integer weights, can be written as $\exp\bigl(\sum u_mx^m/m\bigr)$ for a unique sequence of rational numbers $u_m$.  The series coming from integer weights are exactly those whose coordinates are integers.  This is the reason why integrality of the weights is the right hypothesis, and why the polynomials $C_n$ are \emph{universal}, because they compute the coefficients of every such series from its coordinates.

\begin{definition}[Logarithmic coordinates]\label{def:canonical-coordinates}
For a series $A(x)\in1+x\Q[[x]]$ put $u_m(A)\coloneqq m[x^m]\log A(x)$ for $m\ge1$, so that $A(x)=\exp\bigl(\sum_{m\ge1}u_m(A)x^m/m\bigr)$.
\end{definition}

\begin{proposition}[Classification and universal coefficients]\label{thm:canonical-classification}\label{thm:all-egf-coordinates}
The maps $u\mapsto\exp\bigl(\sum_{m\ge1}u_mx^m/m\bigr)$ and $A\mapsto(u_m(A))_{m\ge1}$ are inverse bijections between the sequences of rational numbers and $1+x\Q[[x]]$, so multiplication of series corresponds to addition of coordinates and the series with integer coordinates are exactly those coming from integer weights.  Moreover, if $A(x)=\sum_na_nx^n/n!$ and $u_m=u_m(A)$, then $a_n=C_n(u_1,\ldots,u_n)$ for every $n\ge0$.
\end{proposition}

\begin{proof}
The formal logarithm and exponential are inverse bijections between $1+x\Q[[x]]$ and $x\Q[[x]]$, and the logarithm turns products into sums.  The formula for $a_n$ is the expansion of Proposition~\ref{prop:cycle-index}, which needs no integrality.
\end{proof}

The polynomials $C_n$ thus encode every series with constant term $1$, and we next find the recurrence they satisfy.

\section{The falling-factorial recurrence}\label{sec:recurrence}

In this section we prove Proposition~\ref{prop:recurrence}, the recurrence from which Theorem~\ref{thm:universal} is deduced in the next section.  To prove a statement about all the numbers $A_w(n)$ at once, we need a way to build $A_w(n+1)$ out of the earlier values using only addition and multiplication.  Then a congruence can be proved by induction on $n$.  Such a recurrence comes from differentiating the generating function.\footnote{The recurrence is the differential equation $F'=F\cdot\sum_jX_jx^{j-1}$ for $F=\exp(\sum_jX_jx^j/j)$, written in coefficients.}  Informally, the idea is to look at the cycle that contains the last label, $n+1$.  Suppose this cycle has length $r+1$.  We choose and order the other $r$ labels on it (the labels that follow $n+1$ along the cycle) in $n(n-1)\cdots(n-r+1)$ ways, give the cycle the weight $w_{r+1}$, and then place any weighted permutation on the remaining $n-r$ labels.  Summing over $r$ gives the recurrence below.

Write $\fall{n}{r}=n(n-1)\cdots(n-r+1)$, with $\fall{n}{0}=1$, for the falling factorial.

\begin{proposition}[Fundamental recurrence]\label{prop:recurrence}
For every $n\ge0$, the universal polynomials satisfy
\begin{equation}\label{eq:recurrence}
 C_{n+1}=\sum_{r=0}^{n}\fall{n}{r}\,X_{r+1}C_{n-r}.
\end{equation}
Consequently, for every integral weight sequence $w$,
\[
 A_w(n+1)=\sum_{r=0}^{n}\fall{n}{r}\,w_{r+1}A_w(n-r).
\]
\end{proposition}

\begin{proof}
Differentiating
\[
 \exp\!\left(\sum_{j\ge1}X_j\frac{x^j}{j}\right)
\]
gives the identity
\[
 \frac{d}{dx}\exp\!\left(\sum_{j\ge1}X_j\frac{x^j}{j}\right)
 =\exp\!\left(\sum_{j\ge1}X_j\frac{x^j}{j}\right)
 \sum_{j\ge1}X_jx^{j-1}.
\]
Extracting the coefficient of $x^n$ and multiplying by $n!$ yields
\eqref{eq:recurrence}.  The formula for $A_w$ follows by setting $X_j=w_j$.
\end{proof}

\begin{example}
For involutions ($w_1=w_2=1$, all other $w_j=0$) only the terms $r=0$ and $r=1$ survive, and the recurrence reads $A(n+1)=A(n)+n\,A(n-1)$.  This gives $A(3)=A(2)+2A(1)=2+2=4$ and $A(4)=A(3)+3A(2)=4+6=10$, matching the table in the Introduction.
\end{example}

Everything about $A_w$ is encoded in this recurrence, and we next study it modulo $k$.

\section{Shift congruences}\label{sec:universal}

In this section we show how Theorem~\ref{thm:universal} follows from the recurrence \eqref{eq:recurrence} by induction on $n$.  Afterwards we draw out what it says about the sequence modulo $k$, namely that it is built from geometric progressions and is periodic whenever $A_w(k)\equiv1\pmod k$.

Roughly speaking, the proof compares the recurrence for $A(n+k+1)$ with the one for $A(n+1)$, which is the same sum except that the falling factorials $\fall{n+k}{r}$ replace $\fall{n}{r}$.  Modulo $k$ the two sets of coefficients agree for $r\le n$, and the extra terms with $r>n$ vanish.  Thus, modulo $k$, the recurrence for $A(n+k+1)$ is the recurrence for $A(n+1)$ with every earlier value shifted by $k$, and an induction closes the argument.

\begin{example}
Take involutions and $k=3$.  Then $A(3)=4\equiv1\pmod3$, and the theorem says $A(n+3)\equiv A(n)\pmod3$.  Indeed $A(3),A(4),A(5)=4,10,26\equiv1,1,2\equiv A(0),A(1),A(2)\pmod 3$.
\end{example}

The two facts used are in the next lemma: part~(i) says that the coefficients agree modulo $k$, and part~(ii) that the extra terms vanish.

\begin{lemma}\label{lem:falling}
For integers $n\ge0$, $k\ge1$, and $r\ge0$:
\begin{enumerate}[label=(\roman*)]
\item if $0\le r\le n$, then $\fall{n+k}{r}\equiv\fall{n}{r}\pmod{k}$,
\item if $n<r\le n+k$, then $k\mid\fall{n+k}{r}$.
\end{enumerate}
\end{lemma}

\begin{proof}
Part (i) follows because $\fall{x}{r}$ is an integer polynomial in $x$.\footnote{If $x\equiv y\pmod k$ then $f(x)\equiv f(y)\pmod k$ for every $f\in\Z[x]$.}  For (ii), the product
\[
 \fall{n+k}{r}=(n+k)(n+k-1)\cdots(n+k-r+1)
\]
contains the factor $k$, namely the term obtained after subtracting $n$, because $r\ge n+1$.
\end{proof}

\begin{proof}[Proof of the modulus-$k$ assertion in Theorem~\ref{thm:universal}]
Work first in $R=\Z[X_1,X_2,\ldots]$ and write\linebreak[1] $C_m=C_m(X_1,\ldots,X_m)$.  Proposition \ref{prop:recurrence}, with $w_j=X_j$, is an identity in $R$.  Fix $k$ and argue by induction on $n$.  For $n=0$ the assertion is tautological because $C_0=1$.  Assume it holds for all indices up to $n$.  Applying the recurrence \eqref{eq:recurrence} with $n+k$ in place of $n$ gives
\begin{equation}\label{eq:shift-recurrence}
 C_{n+k+1}=\sum_{r=0}^{n+k}\fall{n+k}{r}X_{r+1}C_{n+k-r}.
\end{equation}
By Lemma \ref{lem:falling}(ii), the terms of \eqref{eq:shift-recurrence} with $r>n$ vanish in $R/kR$.  For $0\le r\le n$, Lemma \ref{lem:falling}(i) and the induction hypothesis give
\begin{equation}\label{eq:shift-terms}
 \fall{n+k}{r}C_{n+k-r}
 \equiv
 \fall{n}{r}C_{n-r}C_k\pmod{kR}.
\end{equation}
Thus, by \eqref{eq:shift-recurrence}, \eqref{eq:shift-terms} and then \eqref{eq:recurrence},
\begin{equation}\label{eq:shift-conclusion}
 C_{n+k+1}
 \equiv C_k\sum_{r=0}^{n}\fall{n}{r}X_{r+1}C_{n-r}
 =C_kC_{n+1}\pmod{kR},
\end{equation}
which completes the induction.  Substitution $X_j=w_j$ gives the asserted congruence for $A_w$.
\end{proof}

\begin{proof}[Completion of the symmetric strengthening in Theorem~\ref{thm:universal}]
The preceding argument proves divisibility by $k$.  Interchanging $n$ and $k$
proves divisibility by $n$, because $C_nC_k=C_kC_n$.  Coefficientwise
divisibility by both $n$ and $k$ is equivalent to divisibility by
$\operatorname{lcm}(n,k)$.  This completes the proof of Theorem~\ref{thm:universal}.
\end{proof}

\begin{remark}[A second proof, by a group action]\label{rem:group-action}
The proof above is an induction on the recurrence.  A second proof explains where the product $A_w(n)A_w(k)$ comes from, and it is the argument of~\cite[Proposition~2.2]{KallatBala}, stated here for arbitrary weights.

The key point is that a cycle running between the two parts cannot be fixed by a nontrivial rotation.  We first set up the notation.

\emph{Mixed and unmixed permutations.}  Split the $n+k$ points into $P=\{1,\ldots,n\}$ and a set $Q$ of $k$ further points.  Call a permutation of $P\cup Q$ \emph{unmixed} if every cycle lies inside $P$ or inside $Q$, and \emph{mixed} if some cycle meets both.  An unmixed permutation is a pair (a permutation of $P$, a permutation of $Q$), and its weight is the product of the two weights, so the unmixed permutations have total weight exactly $A_w(n)A_w(k)$.  It remains to show that the mixed permutations have total weight divisible by $k$.

\emph{A free action.}  Let the cyclic group $\Z/k$ act on the permutations of $P\cup Q$ by conjugation with powers of a $k$-cycle $\rho$ that rotates the points of $Q$ and fixes the points of $P$.  Conjugation does not change the cycle type, so it does not change the weight monomial.  We claim that no non-trivial rotation $g=\rho^j$ ($0<j<k$) fixes a mixed permutation $\sigma$.  A mixed cycle of $\sigma$ must at some step pass from a point $u\in P$ to a point $\sigma(u)\in Q$.  If $g\sigma g^{-1}=\sigma$, then $g(\sigma(u))=\sigma(g(u))=\sigma(u)$, because $g$ fixes $u$.  So $g$ fixes a point of $Q$, but a non-trivial power of a $k$-cycle has no fixed point.

\emph{Conclusion.}  Hence every orbit of mixed permutations has exactly $k$ elements, all with the same weight monomial, so the mixed permutations contribute $k$ times a polynomial with integer coefficients, and $C_{n+k}-C_nC_k\in k\,\Z[X_1,X_2,\ldots]$.  Exchanging $n$ and $k$ gives divisibility by $n$ as well.  The argument works for every $k$, composite or not.
\end{remark}

\begin{corollary}[Iterated shifts]\label{cor:iterated}
For $q,n\ge0$ and $k\ge1$,
\[
 A_w(n+qk)\equiv A_w(n)A_w(k)^q\pmod{k}.
\]
\end{corollary}

\begin{proof}
Iterate Theorem \ref{thm:universal}.
\end{proof}

\begin{corollary}[Dynamics on residue classes]\label{cor:dynamics}
Fix $k\ge1$ and put $c\equiv A_w(k)\pmod{k}$.  For each residue $0\le r<k$ and every $q\ge0$,
\[
 A_w(r+qk)\equiv A_w(r)c^q\pmod{k}.
\]
Consequently:
\begin{enumerate}[label=(\roman*)]
\item if $c$ is a unit modulo $k$, then $A_w(n)$ is periodic modulo $k$, with period dividing $k\,\operatorname{ord}_k(c)$, where $\operatorname{ord}_k(c)$ is the least $h\ge1$ with $c^h\equiv1\pmod k$,
\item if $c$ is nilpotent modulo $k$ (some power of $c$ is $0$ modulo $k$), then every residue-class subsequence $A_w(r+qk)$ is eventually zero modulo $k$,
\item if $c\equiv1\pmod{k}$, then $k$ itself is a period modulo $k$,
\item if $c\equiv0\pmod{k}$, then $A_w(n)\equiv0\pmod{k}$ for every $n\ge k$.
\end{enumerate}
\end{corollary}

\begin{proof}
The displayed formula is Corollary \ref{cor:iterated}.  If $c$ is a unit, choose its multiplicative order $h$ modulo $k$.  Replacing $q$ by $q+h$ leaves every residue-class subsequence unchanged, so $kh$ is a period.  The other assertions follow by taking a sufficiently large power of a nilpotent $c$, or by specializing to $c=1$ and $c=0$.
\end{proof}

\begin{example}
For involutions, $A(3)=4\equiv1\pmod3$, so (iii) gives period $3$ modulo $3$, with the values $1,1,2,1,1,2,\ldots$.  Modulo $4$ we have $c=A(4)=10\equiv2$, and $2^2\equiv0\pmod4$, so $c$ is nilpotent and (ii) says that every residue-class subsequence is eventually $0$, and indeed the sequence is $1,1,2,0,2,2,0,0,0,\ldots$.
\end{example}

\begin{remark}[Fixed modulus]\label{thm:fixed-modulus-classification}
For any integer sequence $(a_n)$ with $a_0=1$ and fixed $k$, the congruence $a_{n+k}\equiv a_na_k\pmod k$ for all $n\ge0$ is equivalent to $a_{r+qk}\equiv a_ra_k^q\pmod k$ for all $0\le r<k$, $q\ge0$, that is, to each residue-class subsequence being a geometric progression with ratio $a_k$ modulo $k$.  (One direction is an induction on $q$, and the other follows by writing $n=r+qk$.)  Thus the locus of sequences satisfying the shift congruence modulo $k$ is determined by $a_0,\ldots,a_k$ alone, and the content of our theorems is that weighted-permutation sequences lie in it for every $k$.
\end{remark}

\begin{remark}[Substitutions]\label{rem:substitutions}
A coefficientwise congruence survives every substitution of the variables.  For instance, with an extra variable $z$ that counts cycles, $P_n(z)=C_n(zw_1,\ldots,zw_n)$ satisfies $P_{n+k}(z)\equiv P_n(z)P_k(z)\pmod k$ coefficientwise in $z$, which is how the shift congruence applies to triangles of numbers counted by the number of cycles.
\end{remark}

\subsection{The value at a prime, and when \texorpdfstring{$A_w(k)\equiv1$}{A(k)=1}}\label{sec:a-k}

Section~\ref{sec:odd} needs the value of $A_w(p)$ modulo a prime $p$.  For a prime $p$ put $\gamma_p\coloneqq w_1+(p-1)w_p$.

\begin{proposition}\label{prop:prime}
For every prime $p$,
\[
 A_w(p)\equiv w_1^p+(p-1)!w_p\equiv w_1-w_p\equiv\gamma_p\pmod p.
\]
\end{proposition}

\begin{proof}
In \eqref{eq:cycle-index}, every cycle type other than $(1^p)$ and $(p)$ has conjugacy-class size divisible by $p$.  The two exceptional contributions are $w_1^p$ and $(p-1)!w_p$.  Fermat's and Wilson's theorems give the second congruence, and $-w_p\equiv(p-1)w_p\pmod p$ the last.
\end{proof}

The shift congruence is simplest when $A_w(k)\equiv1\pmod k$.  The next lemma says when this happens for the large class of e.g.f.s of the form $F\exp(xG)$.

\begin{lemma}\label{lem:a-k-mod-k}
Let $\sum_{n\ge0}a(n)x^n/n!=F(x)\exp(xG(x))$ with $F,G\in\Z[[x]]$.  Then $a(k)\equiv F(0)G(0)^k\pmod k$ for every $k\ge1$.  In particular, if $F(0)=G(0)=1$ then $a(k)\equiv1\pmod k$.  This applies to every integer weight sequence with $w_1=1$ and $m\mid w_m$ for all $m$, because then $\sum_nA_w(n)x^n/n!=\exp(\sum_mu_mx^m)=\exp(xG(x))$ with $u_m=w_m/m$ and $G(x)=\sum_{m\ge0}u_{m+1}x^m$, $G(0)=u_1=w_1$.
\end{lemma}

\begin{proof}
Write $a(k)=\sum_{j=0}^kf_j\,k!\,[x^{k-j}]e^{xG(x)}$, where $F=\sum f_jx^j$.  For $0\le i\le k$ we have $i!\,[x^i]e^{xG}=\sum_{l=0}^i\frac{i!}{l!}[x^{i-l}]G^l=:\eta_i\in\Z$, so $k!\,[x^{k-j}]e^{xG}=\frac{k!}{(k-j)!}\eta_{k-j}$.  For $j\ge1$ the factor $k!/(k-j)!=k(k-1)\cdots(k-j+1)$ is divisible by $k$.  For $j=0$ we get $f_0\eta_k$, and $\eta_k=\sum_{l=0}^k\frac{k!}{l!}[x^{k-l}]G^l\equiv G(0)^k\pmod k$, because $k!/l!$ is divisible by $k$ for $l<k$.
\end{proof}

\begin{remark}[Bala's class]
Let $F(0)=1$ and $G(0)=1$ in Lemma~\ref{lem:a-k-mod-k}.  Then every logarithmic weight is an integer, since $m[x^m]\log F=[x^m]\,xF'/F\in\Z$ ($F$ is invertible in $\Z[[x]]$) and the factor $\exp(xG)$ contributes $m[x^{m-1}]G\in\Z$.  So the sequences of Bala's 2017 theorem~\cite{Bala2017} are weighted-permutation sequences with $A_w(k)\equiv1\pmod k$, and Theorem~\ref{thm:universal} gives his congruence $a(n+k)\equiv a(n)\pmod k$ back.  Conversely, for general integer weights $A_w(k)\not\equiv1\pmod k$ (for involutions, $A_w(2)=2$), so the factor $A_w(k)$ cannot be dropped.  By Corollary~\ref{cor:dynamics}, $A_w(k)\equiv1\pmod k$ for every $k$ holds exactly when the sequence is purely periodic modulo every $k$ with period dividing $k$.  It would be interesting to describe all weight sequences with this property, and Bala~\cite{Bala2017} asks whether such sequences exist outside the form $F\exp(xG)$.
\end{remark}

\begin{corollary}\label{cor:n-n-1}
Let $w$ be an integral weight sequence with $w_1=1$ and $A_w(k)\equiv1\pmod k$ for all $k$ (by Lemma~\ref{lem:a-k-mod-k} this holds, for instance, when $m\mid w_m$ for all $m$, and more generally when $\sum A_w(n)x^n/n!=F(x)\exp(xG(x))$ with $F(0)=G(0)=1$).  Then $A_w(n)\equiv1\pmod{n(n-1)}$ for all $n\ge2$.  For example, for ``sets of lists''\footnote{The OEIS sequence \OEIS{A000262}.} ($w_m=m$) we have $13\equiv1\pmod 6$, $73\equiv1\pmod{12}$ and $501\equiv1\pmod{20}$.
\end{corollary}

\begin{proof}
Theorem~\ref{thm:universal} with $n\to1$, $k\to n-1$ gives $A_w(n)\equiv A_w(1)A_w(n-1)=A_w(n-1)\equiv1\pmod{n-1}$, and $A_w(n)\equiv1\pmod n$ by hypothesis.  Since $\gcd(n,n-1)=1$, the claim follows.
\end{proof}

The corollary settles the observations of Bala recorded for \OEIS{A251568} and \OEIS{A251569}, namely that $a(n)-1$ is divisible by $n(n-1)$ for $n\ge2$ (their e.g.f.s are $\exp(xK^2)$ and $\exp(xG)$ with $K=1+xK^2$ and $G=1+xG^3$, so Lemma~\ref{lem:a-k-mod-k} applies with $F=1$).

\section{The best modulus for the shift congruence}\label{sec:content}

In this section we prove Theorem~\ref{thm:optimal-divisor}, which gives the exact divisor of the shift congruence, and it rests on Theorem~\ref{thm:universal}, on Lemma~\ref{lem:mixed-support-factorial} and on an orbit count.

Theorem~\ref{thm:universal} says that the \emph{defect}
\begin{equation}\label{eq:defect}
 D_{n,k}\coloneqq C_{n+k}-C_nC_k
\end{equation}
has every coefficient divisible by $\operatorname{lcm}(n,k)$.  The true divisibility can be larger, as two small cases show.
\begin{align*}
 C_4-C_2^2&=4X_1^2X_2+2X_2^2+8X_1X_3+6X_4,\\
 C_5-C_2C_3&=6X_1^3X_2+12X_1X_2^2+18X_1^2X_3+18X_2X_3+30X_1X_4+24X_5.
\end{align*}
In the first case ($n=k=2$) the largest integer dividing every coefficient is $2=\operatorname{lcm}(2,2)$, while in the second ($n=2$, $k=3$) it is $6=\operatorname{lcm}(2,3)=nk$.  For $n=k=4$ it is $8$, not $\operatorname{lcm}(4,4)=4$.  The reader may wish to keep the model case $n=k=2$ in mind throughout this section.

\begin{definition}[Content of the defect]\label{def:defect-content}
The \emph{content} of a polynomial with integer coefficients is the greatest common divisor of its coefficients (Section~\ref{sec:notation}).  For $n,k\ge1$ let $d_{n,k}$ be the content of $D_{n,k}$.  Thus $d_{n,k}$ is the largest positive integer $m$ with $C_{n+k}\equiv C_nC_k\pmod m$ coefficientwise.
\end{definition}

Since reducing coefficients modulo $m$ is a ring homomorphism, $m\mid d_{n,k}$ if and only if $c_{n+k}-c_nc_k\in mR$ for every commutative ring $R$ and every choice of weights $r_j\in R$ (use $R=\Z[X_1,X_2,\ldots]/m$ for the converse).  So $d_{n,k}$ is the best modulus for weights in an arbitrary commutative ring.  We prove
\begin{equation}\label{eq:content-formula}
 d_{n,k}=\frac{nk}{\operatorname{rad}(\gcd(n,k))}.
\end{equation}
In words, if $n$ and $k$ share no prime factor then $d_{n,k}=nk$, and for every prime $p$ dividing both, one factor of $p$ is lost compared with $nk$.

\begin{theorem}[Exact universal divisor]\label{thm:optimal-divisor}
For every $n,k\ge1$,
\[
 d_{n,k}=\frac{nk}{\operatorname{rad}(\gcd(n,k))}.
\]
Equivalently,
\[
 v_p(d_{n,k})=
 \begin{cases}
 v_p(n)+v_p(k)-1,&p\mid\gcd(n,k),\\
 \max\{v_p(n),v_p(k)\},&p\nmid\gcd(n,k).
 \end{cases}
\]
\end{theorem}

\subsection{Outline of the proof}

The proof of \eqref{eq:content-formula} has two parts.  To show that nothing larger works, we exhibit one explicit coefficient of $D_{n,k}$ whose divisibility by each prime is exactly what the formula predicts (Propositions~\ref{prop:content-bounds} and~\ref{prop:prime-witness}).  To show that the predicted number really divides everything, we write each coefficient as a count of permutations, let a product of two symmetric groups act on them by conjugation, and show that all orbits have size divisible by the predicted number (Lemma~\ref{lem:mixed-support-factorial} and the proof of the theorem).

\begin{proposition}[Unconditional content bounds]\label{prop:content-bounds}
For all $n,k\ge1$,
\[
 \operatorname{lcm}(n,k)\mid d_{n,k}\mid nk.
\]
In particular, if $(n,k)=1$, then $d_{n,k}=nk$.
\end{proposition}

\begin{proof}
The left divisibility is the symmetric strengthening of
Theorem~\ref{thm:universal}.  For the right divisibility, inspect the
coefficient of $X_1^{n+k-2}X_2$.  In $C_{n+k}$ it is
$\binom{n+k}{2}$, whereas in $C_nC_k$ it is
$\binom n2+\binom k2$.  Their difference is $nk$.  Since the content divides
every coefficient, $d_{n,k}\mid nk$.  If $(n,k)=1$, then
$\operatorname{lcm}(n,k)=nk$, so both bounds coincide.
\end{proof}

\begin{remark}[Polynomials versus integer weights]\label{rem:poly-vs-integer}
Theorem~\ref{thm:optimal-divisor} is about the coefficients of a polynomial.  A polynomial can be divisible by $m$ at every integer point without being divisible coefficientwise ($X^2-X$), so for integer weights the exact modulus can be larger than the content.  This happens, but rarely, as for $(n,k)=(4,8)$, where the content is $16$ and yet $A_w(12)\equiv A_w(4)A_w(8)\pmod{32}$ for every integer weight sequence, and for $(8,8)$, where the content is $32$ but the congruence holds modulo $64$.\footnote{This was checked exactly by expanding in the basis of products of binomial coefficients, in which divisibility at every integer point is equivalent to divisibility of every coefficient.}  Among pairs with $n\le k$ and $n+k\le20$ the exceptions are exactly $(4,8)$, $(8,8)$, $(4,12)$, $(4,16)$ and $(8,12)$, so the modulus $nk/\operatorname{rad}(\gcd(n,k))$ is attained by integer weights for most pairs $(n,k)$.  It would be interesting to find the exact modulus for integer weights.
\end{remark}

To show that nothing larger works, we look at the single coefficient of $X_{2p}X_1^{N-2p}$, which counts how a $2p$-cycle can be cut between the two sets.

\begin{proposition}[Explicit prime-adic sharpness witness]\label{prop:prime-witness}
Let $p$ be a prime dividing both $n$ and $k$, and put $N=n+k$. For
$\lambda_p=(2p,1^{N-2p})$ one has
\[
 v_p\!\left([X_{2p}X_1^{N-2p}](C_N-C_nC_k)\right)
 =v_p(n)+v_p(k)-1.
\]
Thus no universal content can contain a higher power of a common prime $p$
than the one predicted by Theorem~\ref{thm:optimal-divisor}.
\end{proposition}

\begin{proof}
Write $a=v_p(n)$ and $b=v_p(k)$. Directly selecting how many points of the
unique $2p$-cycle lie in the $n$-part gives
\begin{equation}\label{eq:witness-sum}
 D_p=[X_{2p}X_1^{N-2p}](C_N-C_nC_k)
 =\frac1{2p}\sum_{i=1}^{2p-1}\binom{2p}{i}(n)_i(k)_{2p-i}.
\end{equation}
We first treat the case that $p$ is odd. For $1\le i\le2p-1$,
$v_p\binom{2p}{i}=0$ when $i=p$ and equals $1$ otherwise. Every falling
factorial in \eqref{eq:witness-sum} contains its initial factor, so every noncentral term,
after division by $2p$, has valuation at least $a+b$. The central term has
valuation exactly
\[
 -1+v_p((n)_p)+v_p((k)_p)=a+b-1,
\]
because among $n,n-1,\ldots,n-p+1$ only $n$ is divisible by $p$, and
similarly for $k$. It is therefore the unique term of minimal valuation.

For $p=2$, the central term $i=2$ has valuation $a+b-1$. In the term $i=1$,
$(k)_3=k(k-1)(k-2)$ has valuation at least $b+1$, and the term $i=3$ is
symmetric, so both outer terms consequently have strictly larger valuation.
Hence $v_p(D_p)=a+b-1$ in all cases.
\end{proof}

The next inequality is the arithmetic heart of the lower bound.  Informally, it says that a binomial coefficient and a falling factorial, taken together, cannot have small $p$-adic valuation when $n$ and $k$ are both divisible by $p$.  It will be applied with $i$ and $j$ equal to the numbers of points of the two sets lying on cycles that mix them.  For instance, for $n=k=2$, $p=2$, $i=j=2$ the left side is $v_2(2!\,2!)-1=1=a+b-1$.  Its proof is elementary (Legendre's formula) and is given in Appendix~\ref{app:technical}.

\begin{lemma}[Mixed-support factorial inequality]\label{lem:mixed-support-factorial}
Let $p$ be prime, let $p^a\mid n$ and $p^b\mid k$, and let
$1\le i\le n$, $1\le j\le k$.  Then
\[
 v_p\!\left(\frac{n!k!}{(n-i)!(k-j)!}\right)
 -\min\{v_p(i!),v_p(j!)\}\ge a+b-1.
\]
\end{lemma}

\begin{proof}[Proof of Theorem~\ref{thm:optimal-divisor}]
\emph{Idea.}  Each coefficient of the defect \eqref{eq:defect} counts permutations of $n+k$ points that mix a set $P$ of $n$ points with a set $Q$ of $k$ points.  The group $S_P\times S_Q$ acts on these permutations by conjugation, and the size of an orbit is the order of the group divided by the order of the centralizer of the permutation.  A permutation that mixes $P$ and $Q$ has a small centralizer, because an element of the centralizer is determined by what it does on a few points of the mixed cycles.  Lemma~\ref{lem:mixed-support-factorial} then shows that every orbit has size divisible by the required power of $p$.

The upper bound prime by prime is supplied by
Proposition~\ref{prop:prime-witness} for primes common to $n$ and $k$, and by
Proposition~\ref{prop:content-bounds} for primes dividing only one index.
It remains to prove the lower bound at a common prime $p$.  We turn to the details.

Fix a set $\Omega=P\sqcup Q$ with $|P|=n$, $|Q|=k$, and put
$H=S_P\times S_Q$.  For each cycle type $\lambda\vdash n+k$, the coefficient
$[X_\lambda](C_{n+k}-C_nC_k)$ is the number of permutations of type
$\lambda$ that do not preserve $P$ and $Q$ separately, because $[X_\lambda]C_{n+k}$ counts all permutations of $\Omega$ of type $\lambda$, while $[X_\lambda](C_nC_k)$ counts the pairs of a permutation of $P$ and a permutation of $Q$ whose cycle types combine to $\lambda$, that is, the permutations of type $\lambda$ that preserve both $P$ and $Q$.  The set of non-preserving permutations is stable
under conjugation by $H$.

For such a permutation $\sigma$, let $M$ be the union of its cycles meeting
both $P$ and $Q$, and put $i=|M\cap P|$, $j=|M\cap Q|$.  Then
$1\le i\le n$ and $1\le j\le k$.  An element of the centralizer $Z_H(\sigma)$ commutes with $\sigma$, so its action on a mixed cycle is determined by where it sends one point of that cycle.  Hence its restriction to the mixed cycles is determined by its action on $M\cap P$ (which has $i$ points), and also by its action on $M\cap Q$ (which has $j$ points), so the group of these restrictions has order dividing both $i!$ and $j!$, and its $p$-adic valuation is at most $\min\{v_p(i!),v_p(j!)\}$.  The remaining part of the centralizer acts on the complement of $M$, which has $n-i$ points in $P$ and $k-j$ points in $Q$, and contributes at most $(n-i)!(k-j)!$.  Hence
\begin{equation}\label{eq:centralizer-bound}
 v_p(|Z_H(\sigma)|)
 \le v_p((n-i)!)+v_p((k-j)!)
 +\min\{v_p(i!),v_p(j!)\}.
\end{equation}
Lemma~\ref{lem:mixed-support-factorial} applied to \eqref{eq:centralizer-bound} therefore gives
\begin{equation}\label{eq:orbit-bound}
 v_p(|H:Z_H(\sigma)|)\ge v_p(n)+v_p(k)-1.
\end{equation}
By \eqref{eq:orbit-bound}, every $H$-orbit in the coefficient count has this divisibility, so every
coefficient of the defect is divisible by the predicted common-prime power.
Together with the already sharp upper bounds, this proves
\eqref{eq:content-formula}.

\emph{A small case.}  For $n=k=2$ take $P=\{1,2\}$, $Q=\{3,4\}$ and $\sigma=(1\,3)(2\,4)$, so $M=\Omega$ and $i=j=2$.  Here $H=\{e,(12),(34),(12)(34)\}$ has order $4$, and only $e$ and $(12)(34)$ commute with $\sigma$, so the orbit of $\sigma$ has size $2=2^{1+1-1}$, as predicted.  This completes the proof of Theorem~\ref{thm:optimal-divisor}.
\end{proof}

\section{Lifting the shift congruence at every prime}\label{sec:lift}

In this section we prove Theorem~\ref{thm:strong-all-primes}, and it depends on the exact divisor of Theorem~\ref{thm:optimal-divisor} for pairs of blocks and on Lemma~\ref{lem:connected-divisible} for connected pieces.

The shift congruence says that $A(n+k)\equiv A(n)A(k)\pmod k$.  Taking $n=k$ and repeating, we get $A(pk)\equiv A(k)^p\pmod k$ for every $p$, so going from $k$ to $pk$ acts like raising to the $p$-th power, at least modulo $k$.

We now examine how far this persists when $p$ divides $k$.  For the involution numbers $1,1,2,4,10,26,76,232,764,2620,\ldots$ with $p=3$ and $k=3$ we have $A(9)=2620$ and $A(3)^3=4^3=64$, so
\[
 A(9)-A(3)^3=2556=3^2\cdot 284 .
\]
The shift congruence only promises divisibility by $k=3$, but the difference is divisible by $3^2$.  For $p=2$ and $k=2$ we get $A(4)-A(2)^2=10-4=6$, divisible by $2$.  The general rule is the following (Theorem~\ref{thm:intro-lift}, restated).

\begin{theorem}[Lifting at every prime]\label{thm:strong-all-primes}
Let $p$ be a prime, let $k\ge1$ be divisible by $p$, and put $s=v_p(k)$.  Then every coefficient of the polynomial $C_{pk}-C_k^{\,p}$ is divisible by
\[
 p^{2s}\quad\text{if }p\text{ is odd},\qquad p^{2s-1}\quad\text{if }p=2 .
\]
Consequently, for every integral weight sequence $w$,
\[
 A_w(pk)\equiv A_w(k)^p\pmod{p^{2s}}\ (p\text{ odd}),\qquad
 A_w(2k)\equiv A_w(k)^2\pmod{2^{2s-1}} .
\]
\end{theorem}

For involutions, the theorem gives $9\mid A(9)-A(3)^3$ and (with $p=2$, $k=4$) $2^3\mid A(8)-A(4)^2=664$, both of which are true.

\subsection{Idea of the proof}

Think of $pk$ points arranged in $p$ blocks of $k$ points each.  The number $A(k)^p$ counts permutations that keep every point inside its own block (this is the product $C_k^p$).  The number $A(pk)$ counts all permutations, so $A(pk)-A(k)^p$ counts the permutations that move at least one point from one block to another.

We split such a permutation into independent pieces.  Draw a graph whose vertices are the $p$ blocks, joining two blocks whenever the permutation sends some point of one block into the other.  The connected components of this graph cut the permutation into pieces that do not interact.  A piece that involves $j$ blocks and is connected is called a \emph{connected permutation of $j$ blocks}.  It turns out that:
\begin{itemize}
\item a connected piece involving $j\ge3$ blocks always contributes a multiple of $k^{2}$ (in fact $k^{j-1}$), because the rotations inside each block act on such pieces almost freely,
\item the only pieces that can contribute less are pairs of blocks, and these are measured exactly by the two-term defect $C_{2k}-C_k^2$ of Section~\ref{sec:content},
\item there are $\binom p2$ ways to choose which two blocks interact, and $\binom p2$ has one extra factor of $p$ when $p$ is odd.  That extra factor is the whole difference between odd primes and $2$.
\end{itemize}
The model case to consider is $k=p$, where $s=1$ and the theorem asserts divisibility of $C_{p^2}-C_p^{\,p}$ by $p^2$ for odd $p$ (as in the case $p=k=3$ above).  We now make this precise.

\subsection{Connected permutations of blocks}

For a finite set $S$ and a permutation $g$ of $S$, write
\[
 \operatorname{wt}(g)=\prod_{c\text{ a cycle of }g}X_{|c|}\in\Z[X_1,X_2,\ldots]
\]
for its weight monomial.  For example, if $g=(1\,2)(3)$ then $\operatorname{wt}(g)=X_2X_1$.
By Proposition~\ref{prop:cycle-index}, $C_n=\sum_{g\in\mathrm{Sym}(S)}\operatorname{wt}(g)$ for any set $S$ with $n$ elements.

Fix $k\ge1$ and $j\ge1$.  Let $S_{j,k}=\{1,\ldots,j\}\times\Z/k$, which we think of as $j$ blocks of $k$ points.  For a permutation $g$ of $S_{j,k}$, the \emph{block graph} of $g$ has the vertices $1,\ldots,j$, and it joins $a\ne b$ if $g$ sends some point $(a,x)$ to a point $(b,y)$.  We call $g$ \emph{connected} if its block graph is connected, and we set
\[
 \mathrm{Conn}_j(k)=\sum_{g\text{ connected}}\operatorname{wt}(g)\in\Z[X_1,X_2,\ldots].
\]
For $j=1$ every permutation of the single block is connected, so $\mathrm{Conn}_1(k)=C_k$.

\begin{example}
Let $p=3$ and $k=2$, with three blocks $\{1,2\}$, $\{3,4\}$, $\{5,6\}$.  The permutation $g=(1\,3)(2\,4)$ fixes $5$ and $6$ and swaps the first two blocks, so its block graph has a single edge, between blocks $1$ and $2$.  It is therefore a connected permutation of those two blocks, followed by a (trivial) connected permutation of the third block, and $\operatorname{wt}(g)=X_2^2X_1^2$.  A permutation such as $(1\,3\,5)(2\,4\,6)$ moves points from every block to another, so its block graph is connected on all three blocks.
\end{example}

\begin{lemma}[Cutting a permutation into connected pieces]\label{lem:connected-decomposition}
For all $m,k\ge1$,
\begin{equation}\label{eq:connected-decomposition}
 C_{mk}=\sum_{\pi}\ \prod_{B\in\pi}\mathrm{Conn}_{|B|}(k),
\end{equation}
where $\pi$ runs over all set partitions of the $m$ blocks $\{1,\ldots,m\}$.  In particular
\begin{equation}\label{eq:conn-two}
 C_{2k}=C_k^2+\mathrm{Conn}_2(k),\qquad\text{so}\qquad \mathrm{Conn}_2(k)=C_{2k}-C_k^2 .
\end{equation}
\end{lemma}

\begin{proof}
Let $g$ be any permutation of the $m$ blocks of size $k$.  The connected components of its block graph form a set partition $\pi$ of $\{1,\ldots,m\}$.  Since $g$ never sends a point outside the union of the blocks in a component, $g$ restricts to a permutation of each of these unions, and the restriction is connected.  Conversely, choosing a connected permutation of each union produces a permutation whose block graph has exactly the components $\pi$.  Finally, $\operatorname{wt}$ is multiplicative over pieces, and $\operatorname{wt}$ of a connected permutation of $|B|$ blocks does not depend on how the blocks are named.  Summing over $g$ gives the formula.
\end{proof}

\begin{lemma}[Connected pieces are highly divisible]\label{lem:connected-divisible}
Every coefficient of $\mathrm{Conn}_j(k)$ is divisible by $k^{j-1}$.
\end{lemma}

\begin{proof}
Let $G=(\Z/k)^j$ act on $S_{j,k}$ by rotating each block, so that $\rho=(\rho_1,\ldots,\rho_j)$ sends $(a,x)$ to $(a,x+\rho_a)$.  Conjugation by $\rho$ preserves the block graph and the cycle lengths, so $G$ acts on the set of connected permutations and $\operatorname{wt}$ is constant on each orbit.  It therefore suffices to show that every orbit has size divisible by $k^{j-1}$.

Let $g$ be connected and let $\rho$ commute with $g$ (that is, $\rho$ lies in the stabilizer of $g$).  Suppose $g(a,x)=(b,y)$ with $b\ne a$.  Then $\rho\,g(a,x)=g\,\rho(a,x)$ reads $(b,y+\rho_b)=g(a,x+\rho_a)$.  If $\rho_a=0$, the right side is $g(a,x)=(b,y)$, so $\rho_b=0$ as well.  The same argument applied to $g^{-1}$ (which also commutes with $\rho$) shows that $\rho_b=0$ forces $\rho_a=0$.  Thus along every edge of the block graph, one coordinate vanishes if and only if the other does, and since the graph is connected, if one coordinate of $\rho$ vanishes then all do.  Hence the map $\rho\mapsto \rho_1$ from the stabilizer to $\Z/k$ is injective, and the stabilizer has order dividing $k$.  By orbit--stabilizer, the orbit of $g$ has size $k^j/|\mathrm{Stab}|$, which is a multiple of $k^{j-1}$.
\end{proof}

\begin{remark}
For $j=2$ the lemma gives only divisibility of $\mathrm{Conn}_2(k)$ by $k$.  The proof of the theorem uses the sharper divisibility by $k^2/\operatorname{rad}(k)$ from Theorem~\ref{thm:optimal-divisor}, and this is the only place where the exact content of Section~\ref{sec:content} is used.
\end{remark}

\subsection{Proof of the theorem}

\begin{proof}[Proof of Theorem~\ref{thm:strong-all-primes}]
By \eqref{eq:connected-decomposition} with $m=p$, and since the partition into singletons contributes $C_k^p$ by $\mathrm{Conn}_1(k)=C_k$,
\begin{equation}\label{eq:lift-sum}
 C_{pk}-C_k^{\,p}=\sum_{\pi\ne\{\{1\},\ldots,\{p\}\}}\ \prod_{B\in\pi}\mathrm{Conn}_{|B|}(k),
\end{equation}
the sum being over the set partitions $\pi$ of $\{1,\ldots,p\}$ other than the partition into singletons.  Write $\delta=\mathrm{Conn}_2(k)=C_{2k}-C_k^2$, as in \eqref{eq:conn-two}.  By Theorem~\ref{thm:optimal-divisor} with $n=k$, every coefficient of $\delta$ is divisible by $k^2/\operatorname{rad}(k)$, which has $p$-adic valuation $2s-1$.  We sort the partitions $\pi$ in \eqref{eq:lift-sum} into two kinds, and dispose of the easy kind first.

\emph{Some block of $\pi$ has size $j\ge3$.}  By Lemma~\ref{lem:connected-divisible}, the factor $\mathrm{Conn}_j(k)$ is divisible by $k^{j-1}$, hence by $k^2$, which has $p$-adic valuation $2s$.  The other factors are integer polynomials, so this term is divisible by $p^{2s}$.

\emph{Every block of $\pi$ has size $1$ or $2$, with $t\ge1$ blocks of size $2$.}  There are $N_t=\frac{p!}{2^t\,t!\,(p-2t)!}$ such partitions, and each contributes $\delta^t\,C_k^{\,p-2t}$.  If $t\ge2$, then $\delta^t$ is divisible by $p^{2(2s-1)}=p^{4s-2}$, and $4s-2\ge2s$ because $s\ge1$, so these terms are negligible.  If $t=1$, the total is
\begin{equation}\label{eq:lift-main-term}
 \binom p2\,\delta\,C_k^{\,p-2}.
\end{equation}
Here $\binom p2=\tfrac{p(p-1)}2$.  For odd $p$ this has $p$-adic valuation exactly $1$, so \eqref{eq:lift-main-term} is divisible by $p^{1+(2s-1)}=p^{2s}$.  For $p=2$ we have $\binom 22=1$, and \eqref{eq:lift-main-term} is divisible by $2^{2s-1}$.

Every term is therefore divisible by the stated power of $p$.  The statement about $A_w$ follows by evaluating the polynomials at $X_j=w_j$.  This completes the proof of Theorem~\ref{thm:strong-all-primes}.
\end{proof}

\begin{remark}[Sharpness of the exponent]\label{rem:lift-sharp}
The exponent is sharp as a statement about the polynomials.  For $p=2$ and every $s$ this is Theorem~\ref{thm:optimal-divisor}, because $C_{2k}-C_k^2$ is the defect with $n=k$, whose content is exactly $k^2/\operatorname{rad}(k)=2^{2s-1}$.  (For integer weights the true modulus can be larger, as for $k=8$, where $A_w(16)\equiv A_w(8)^2$ holds modulo $2^6$ for every integer weight sequence although the content is $2^5$, as explained in Remark~\ref{rem:poly-vs-integer}.)  For odd $p$ and every $k$ with $p\mid k$ an explicit coefficient has valuation exactly $2s$ (Proposition~\ref{prop:lift-sharp-odd} in Appendix~\ref{app:technical}).  For example, for $p=k=3$ the coefficient of $X_3X_6$ in $C_9-C_3^3$ is $9!/(3\cdot6)=20160=3^2\cdot2240$, of valuation $2=2s$.  We also checked the witness by exact computation for $p=3,5,7$ and every multiple $k\le60$ of $p$.\footnote{The program is \texttt{verify\_lift\_sharpness.py}.}
\end{remark}

The theorem lets us compute $A_w(p^r)$ modulo $p^r$ in closed form (for involutions, $A(3^r)\equiv1\pmod{3^r}$ for every $r$), which is the subject of the next section.

\section{Closed formulas for \texorpdfstring{$A_w(p^r)$}{A(p**r)} modulo \texorpdfstring{$p^r$}{p**r}}\label{sec:odd}

In this section we deduce the closed formulas of Theorem~\ref{thm:intro-closed} from Theorem~\ref{thm:strong-all-primes} by induction on $r$.  By the Chinese remainder theorem, the residue of a number modulo $n$ is determined by its residues modulo the prime powers dividing $n$.  So the most basic question about the sequence $A_w(n)\bmod n$ is what happens at $n=p^r$.  The lifting theorem of the previous section answers it by a simple induction, in which knowing the residue at $p^r$ and taking a $p$-th power gives the residue at $p^{r+1}$, one power of $p$ at a time.

\begin{lemma}\label{lem:pth-power}
Let $p$ be a prime and $r\ge1$.  If $x\equiv y\pmod{p^r}$, then $x^p\equiv y^p\pmod{p^{r+1}}$.
\end{lemma}

\begin{proof}
Write $x=y+p^rt$.  Then $x^p=y^p+p\,y^{p-1}p^rt+\binom p2y^{p-2}p^{2r}t^2+\cdots$, and every term after $y^p$ is divisible by $p^{r+1}$.
\end{proof}

\subsection{Odd primes}

For odd $p$ the lift gains a full extra factor of $p$ at every step, and the induction runs cleanly.

\begin{corollary}[Residue at odd prime powers]\label{cor:odd-prime-power}
Let $p$ be an odd prime, $w$ an integral weight sequence and $r\ge1$.  Then
\[
 A_w(p^r)\equiv\bigl(w_1+(p-1)w_p\bigr)^{p^{r-1}}\pmod{p^r}.
\]
\end{corollary}

\begin{proof}
Put $\gamma\coloneqq w_1+(p-1)w_p$.  We argue by induction on $r$.

\emph{The case $r=1$.}  By Proposition~\ref{prop:prime}, $A_w(p)\equiv w_1-w_p\pmod p$, and $w_1-w_p\equiv w_1+(p-1)w_p=\gamma\pmod p$.\footnote{Only the residue of $\gamma$ modulo $p$ matters for the base, and we write it as $w_1+(p-1)w_p$ so that the formula uses an integer expression.}

\emph{From $r$ to $r+1$.}  Assume
\begin{equation}\label{eq:odd-hypothesis}
 A_w(p^r)\equiv\gamma^{p^{r-1}}\pmod{p^r}.
\end{equation}
By Lemma~\ref{lem:pth-power} applied to \eqref{eq:odd-hypothesis},
\begin{equation}\label{eq:odd-power}
 A_w(p^r)^p\equiv\gamma^{p^r}\pmod{p^{r+1}}.
\end{equation}
Theorem~\ref{thm:strong-all-primes} with $k=p^r$ gives
\begin{equation}\label{eq:odd-lift}
 A_w(p^{r+1})\equiv A_w(p^r)^p\pmod{p^{2r}},
\end{equation}
and $2r\ge r+1$.  Combining \eqref{eq:odd-lift} and \eqref{eq:odd-power}, $A_w(p^{r+1})\equiv\gamma^{p^r}\pmod{p^{r+1}}$.  This completes the proof of the corollary.
\end{proof}

\begin{example}
For involutions ($w_1=w_2=1$, the others $0$) and $p=3$ we get $\gamma=w_1+2w_3=1$, so $A(3^r)\equiv1\pmod{3^r}$ for all $r$.  This matches $A(3)=4$ and $A(9)=2620=3^2\cdot291+1$.
\end{example}

\subsection{The prime \texorpdfstring{$2$}{2}}

At $p=2$ the lift is weaker by one power of $2$ (Theorem~\ref{thm:strong-all-primes} only gives $2^{2s-1}$), and this has a visible effect, because the formula $\gamma^{2^{r-1}}$ with $\gamma=w_1+w_2$ is correct at $r=1$ and at every $r\ge4$, but it is wrong in general at $r=2$ and $r=3$ (it holds only for special weights).  So the first levels must be computed by hand, and the weight $w_4$ enters through $\beta=w_2+w_4$.  The precise statement is as follows.

\begin{theorem}[Explicit power-of-two residue]\label{thm:power-of-two}
Let $w$ be an integral weight sequence and put $\alpha=w_1+w_2$ and $\beta=w_2+w_4$.  Then
\begin{align*}
 A_w(2)&\equiv\alpha &&\pmod 2,\\
 A_w(4)&\equiv\alpha^2+2\beta &&\pmod 4,\\
 A_w(8)&\equiv\alpha^4+4\alpha^2\beta+4\beta^2 &&\pmod 8,
\end{align*}
and for every $r\ge4$
\[
 A_w(2^r)\equiv(w_1+w_2)^{2^{r-1}}\pmod{2^r}.
\]
Consequently $A_w(2^r)\bmod 2^r$ is determined by $w_1$ and $w_2$ for $r=1$ and for all $r\ge4$, and by $w_1,w_2,w_4$ for $r=2,3$.  The bare power $(w_1+w_2)^{2^{r-1}}$ fails at $r=2$ and $r=3$.
\end{theorem}

\begin{example}
For involutions $\alpha=w_1+w_2=2$ and $\beta=w_2+w_4=1$.  The theorem says $A(4)\equiv\alpha^2+2\beta=6\equiv2\pmod4$, and indeed $A(4)=10$.  The bare power $\alpha^2=4\equiv0\pmod4$ would be wrong, which is why the weight $w_4$ (through $\beta$) appears at this level, since the lift loses one power of $2$ here and the first cycle length beyond $1$ and $2$ that can contribute is $4$.
\end{example}

\begin{proof}
A routine calculation handles levels $2$ to $4$, and the idea is the same induction as before, with the first steps done by hand.  The reader may wish to read the odd-prime case, Corollary~\ref{cor:odd-prime-power}, first.

\emph{Levels $r=1,2$.}  By the cycle-index formula, $A_w(2)=w_1^2+w_2$ and $A_w(4)=w_1^4+6w_1^2w_2+3w_2^2+8w_1w_3+6w_4$.  For $r=1$, $A_w(2)-(w_1+w_2)=w_1^2-w_1$ is even.  For $r=2$,
\begin{equation}\label{eq:level-two}
 A_w(4)-\bigl(\alpha^2+2\beta\bigr)=w_1^2(w_1^2-1)+2w_2\,(3w_1^2-w_1-1+w_2)+8w_1w_3+4w_4 .
\end{equation}
In \eqref{eq:level-two}, $w_1^2(w_1^2-1)\equiv0\pmod4$, since $4\mid w_1^2$ for even $w_1$ and $8\mid w_1^2-1$ for odd $w_1$.  Also $3w_1^2-w_1-1+w_2\equiv w_2-1\pmod2$, so $2w_2(\cdots)\equiv2w_2(w_2-1)\equiv0\pmod4$.

\emph{Level $r=3$.}  Theorem \ref{thm:strong-all-primes} (with $p=2$ and $k=4$) gives $A_w(8)\equiv A_w(4)^2\pmod8$.  By Lemma~\ref{lem:pth-power} with $p=2$, $x\equiv y\pmod{2^r}$ implies $x^2\equiv y^2\pmod{2^{r+1}}$.  Since $A_w(4)\equiv\alpha^2+2\beta\pmod4$, we get $A_w(4)^2\equiv(\alpha^2+2\beta)^2=\alpha^4+4\alpha^2\beta+4\beta^2\pmod8$.

\emph{Level $r=4$.}  Theorem \ref{thm:strong-all-primes} (with $p=2$ and $k=8$) gives $A_w(16)\equiv A_w(8)^2\pmod{32}$, and Lemma~\ref{lem:pth-power} with the level-$3$ congruence gives $A_w(8)^2\equiv(\alpha^4+4\alpha^2\beta+4\beta^2)^2\pmod{16}$.  Expanding,
\begin{equation}\label{eq:level-four}
 (\alpha^4+4\alpha^2\beta+4\beta^2)^2\equiv\alpha^8+8\alpha^4\beta(\alpha^2+3\beta)\pmod{16},
\end{equation}
and $\alpha^4\beta(\alpha^2+3\beta)$ is even, which is clear if $\alpha$ is even, while if $\alpha$ is odd it is $\equiv\beta(1+\beta)\pmod2$.  Hence, by \eqref{eq:level-four}, $A_w(16)\equiv\alpha^8\pmod{16}$.

The bare power fails at $r=2,3$, since with $w_4=1$ and all other weights $0$ one has $A_w(4)=6\equiv2\pmod4$ and $A_w(8)=1260\equiv4\pmod8$, while $\alpha=0$.

\emph{From $r$ to $r+1$, for $r\ge4$.}  Let $r\ge4$ and assume $A_w(2^r)\equiv\alpha^{2^{r-1}}\pmod{2^r}$.  Theorem \ref{thm:strong-all-primes} (with $p=2$ and $k=2^r$) gives
\begin{equation}\label{eq:two-lift}
 A_w(2^{r+1})\equiv A_w(2^r)^2\pmod{2^{2r-1}},\qquad 2r-1\ge r+1.
\end{equation}
If $x\equiv y\pmod{2^r}$, then $x^2-y^2=(x-y)(x+y)$ with $x+y=(x-y)+2y$ even, so $x^2\equiv y^2\pmod{2^{r+1}}$.  Hence, by \eqref{eq:two-lift},
\[
 A_w(2^{r+1})\equiv A_w(2^r)^2\equiv\bigl(\alpha^{2^{r-1}}\bigr)^2=\alpha^{2^r}\pmod{2^{r+1}}.
\]
This completes the proof of Theorem~\ref{thm:power-of-two}.
\end{proof}

\begin{corollary}[Powers of two are $0$ or $1$]\label{cor:power-of-two-01}
For every integral weight sequence $w$ and every $r\ge4$,
\[
 A_w(2^r)\equiv\begin{cases}1\pmod{2^r}&\text{if }w_1+w_2\text{ is odd},\\[2pt]0\pmod{2^r}&\text{if }w_1+w_2\text{ is even.}\end{cases}
\]
\end{corollary}

\begin{proof}
By Theorem \ref{thm:power-of-two}, $A_w(2^r)\equiv g^{2^{r-1}}\pmod{2^r}$ with $g=w_1+w_2$.  If $g$ is even, then $2^{2^{r-1}}\mid g^{2^{r-1}}$ and $2^{r-1}\ge r$.  If $g$ is odd, then $g\equiv1\pmod2$, and $x\equiv1\pmod{2^j}$ implies $x^2\equiv1\pmod{2^{j+1}}$ by Lemma~\ref{lem:pth-power}, so $g^{2^{j-1}}\equiv1\pmod{2^j}$ for all $j\ge1$.
\end{proof}

\subsection{Eventual divisibility}

The closed formulas also determine the primes $p$ for which the sequence is eventually divisible by every power of $p$.  The answer depends only on the base $\gamma_p=w_1+(p-1)w_p$ that appears in the residue formula.

\begin{corollary}[Eventual divisibility]\label{cor:eventual-divisibility}
Let $p$ be a prime, $e\ge1$, and $\gamma_p=w_1+(p-1)w_p$.  Then $A_w(n)\equiv0\pmod{p^e}$ for all sufficiently large $n$ if and only if $p\mid\gamma_p$.  In that case, $A_w(n)\equiv0\pmod{p^r}$ for every $n\ge p^r$, whenever $r\ge e$ (and $r\ge4$ if $p=2$).
\end{corollary}

\begin{proof}
Suppose $p\mid\gamma_p$ and choose $r\ge e$, with $r\ge4$ if $p=2$.  By Corollary~\ref{cor:odd-prime-power} (for odd $p$) or Theorem~\ref{thm:power-of-two} (for $p=2$), $A_w(p^r)\equiv\gamma_p^{\,p^{r-1}}\pmod{p^r}$.  Since $p\mid\gamma_p$ and $p^{r-1}\ge r$, the right side is divisible by $p^{p^{r-1}}$, hence by $p^r$, so $A_w(p^r)\equiv0\pmod{p^r}$.  The shift congruence with $k=p^r$ gives $A_w(n+p^r)\equiv A_w(n)A_w(p^r)\equiv0\pmod{p^r}$ for every $n\ge0$, so $A_w(m)\equiv0\pmod{p^r}$ for all $m\ge p^r$.

Conversely, suppose $p\nmid\gamma_p$ and that $A_w(n)\equiv0\pmod p$ for all $n\ge n_0$.  By Proposition~\ref{prop:prime}, $A_w(p)\equiv w_1-w_p\equiv\gamma_p\pmod p$ is a unit modulo $p$, and the shift congruence modulo $p$ gives $A_w(n)\equiv A_w(p)^{-1}A_w(n+p)\pmod p$.  Applying this repeatedly we find $A_w(n)\equiv0\pmod p$ for every $n\ge0$, including $n=0$, which contradicts $A_w(0)=1$.
\end{proof}

\begin{example}
For involutions and $p=2$ we have $\gamma_2=w_1+w_2=2$, which is even, so every power of $2$ divides $A(n)$ for all large enough $n$, and indeed $A(16)=46206736$ is divisible by $16$, and $A(n)\equiv0\pmod{2^r}$ for all $n\ge2^r$ once $r\ge4$.  For $p=3$ we have $\gamma_3=1$, which is not divisible by $3$, so the sequence is never eventually divisible by $3$, and indeed it is periodic modulo $3$ with the values $1,1,2$.
\end{example}

\begin{remark}[Comparison with the literature]\label{rem:BO}
Combined with Corollary~\ref{cor:odd-prime-power}, the above shows that $A_w(p^r)\equiv\bigl(w_1+(p-1)w_p\bigr)^{p^{r-1}}\pmod{p^r}$ for all primes $p$ and all $r\ge1$ except $(p,r)=(2,2),(2,3)$, and the argument needs only Theorem~\ref{thm:strong-all-primes} and the explicit values of $A_w(2)$ and $A_w(4)$.  Bellagh and Oulebsir~\cite{BellaghOulebsir2023} (whose ``cycle indicator'' is our cycle index polynomial) proved $A_w(\rho+qp)\equiv(w_1^p-w_p)^qA_w(\rho)$ modulo $q^*p$ in $\Z_p$, where $q^*=q$ for odd $q$ and $q^*=q/2$ for even $q$.  For odd $p$ the factor $2$ is a unit, so this modulus has the same $p$-adic valuation as $qp$, namely $1+v_p(q)$.  Taking $\rho=0$ and $q=p^{r-1}$ and lifting the base from $w_1^p-w_p$ to $w_1+(p-1)w_p$ by a Frobenius argument also gives Corollary~\ref{cor:odd-prime-power} for odd $p$.  For $p=2$ and $q=2^{r-1}$ the modulus $q^*p=2^{r-1}$ is one power of $2$ weaker than the $2^r$ we prove.  The proof above is self-contained and covers all primes.
\end{remark}

\section{A family of sequences from Lagrange inversion}\label{sec:lagrange}

In this section we show, using Theorem~\ref{thm:lagrange}, that the theorems above also apply to a family that does not look like weighted permutations at first sight, but is one, namely the numbers $n![x^n]\,e^{xG(x)}F(x)^n$ obtained from powers of a series.  This settles a conjecture of Bala.  The applications of Section~\ref{sec:sequences} use Sections~\ref{sec:universal} and~\ref{sec:odd}, and Bala's observations for \OEIS{A251568} and \OEIS{A251569} follow from Corollary~\ref{cor:n-n-1}.

A natural way to produce integer sequences is to take the coefficients of the $n$-th power of a series.  For example, $n!\,[x^n]\,e^x(1-x)^{-n}$ is the OEIS sequence \OEIS{A278070}, and Bala observed from computer data that it satisfies $a(n+k)\equiv a(n)\pmod k$.  He conjectured that the same holds for the whole family below, for all integer series $F$ and $G$ with $F(0)=G(0)=1$.  The family is not in general of the form $F(x)\exp(xG(x))$ treated in~\cite{Bala2017}, because the exponent $n$ sits on $F$.  Our theorem shows that these sequences are, after all, weighted-permutation counts with integer weights, so all the congruences above apply to them.

For integer power series $F,G\in\Z[[x]]$ put
\[
 b(n)=n!\,[x^n]\,e^{xG(x)}F(x)^n .
\]
By Lagrange inversion there is a unique series $x(y)\in\Z[[y]]$ with $x(y)=yF(x(y))$, and it has integer coefficients because $F$ does.

Roughly speaking, Lagrange inversion turns the generating function $\sum_nb(n)y^n/n!$ into an explicit expression in the series $x(y)$.  Its logarithm is then explicit, and the weights are read off from the logarithmic derivative, which involves only derivatives, products and inverses of integer series with constant term $1$, and therefore has integer coefficients.

\begin{theorem}[Lagrange-type sequences]\label{thm:lagrange}
Put $\theta(y)=y\,F'(x(y))\in y\Z[[y]]$ (the derivative of $F$ along $x(y)$, times $y$) and
\[
 \sum_{m\ge1}w_my^m=y\frac{d}{dy}\Bigl(x(y)G(x(y))\Bigr)+\frac{y\,\theta'(y)}{1-\theta(y)} .
\]
Then every $w_m$ is an integer and $\sum_{n\ge0}b(n)\,y^n/n!=\exp\bigl(\sum_{m\ge1}w_my^m/m\bigr)$.  Thus $b(n)=A_w(n)$ for an integral weight sequence $w$, and every result above applies to $b$.  Moreover $b(k)\equiv(F(0)G(0))^k\pmod k$ for all $k\ge1$.  In particular, if $F(0)=G(0)=1$, then
\[
 b(n+k)\equiv b(n)\pmod k\qquad\text{for all }n\ge0,\ k\ge1 .
\]
\end{theorem}

\begin{example}
Take $F=1/(1-x)$ and $G=0$.  Then $b(n)=n!\binom{2n-1}{n}$ for $n\ge1$, giving $1,\ 1,\ 6,\ 60,\ 840,\ 15120,\ \ldots$, and the theorem says that these are the numbers $A_w(n)$ for the weights $w=(1,\ 5,\ 22,\ 93,\ 386,\ 1586,\ \ldots)$.  Indeed $x(y)=y+y^2+2y^3+5y^4+\cdots$ and $\theta(y)=y+2y^2+5y^3+\cdots$, so $y\theta'(y)/(1-\theta(y))=y+5y^2+\cdots$, which gives $w_1=1$ and $w_2=5$.  Here $F(0)G(0)=0$, and indeed $b(k)\equiv0\pmod k$ for every $k$.
\end{example}

\begin{proof}
The second form of Lagrange inversion (Gessel~\cite[Theorem~2.1.1]{Gessel2016}) states that for $H\in\Q[[x]]$,
\[
 \sum_{n\ge0}y^n[x^n]\,H(x)F(x)^n=\frac{H(x(y))}{1-yF'(x(y))}.
\]
(If $F(0)=0$ then $x(y)=0$ and both sides equal $H(0)/(1-yF'(0))$, so there is nothing to prove.)  Take $H=e^{xG(x)}$.  Since $b(n)y^n/n!=y^n[x^n]HF^n$, the left side is $\Psi(y):=\sum_nb(n)y^n/n!$, and
\[
 \log \Psi(y)=x(y)G(x(y))-\log\bigl(1-\theta(y)\bigr).
\]
Applying $y\,d/dy$ gives $y\,(\log \Psi)'=\sum_mw_my^m$ with the $w_m$ as stated.  The first term is the image of an integral series under composition with $x(y)\in y\Z[[y]]$ followed by $y\,d/dy$, hence integral.  In the second term, $\theta\in y\Z[[y]]$ and $1-\theta$ has constant term $1$, so $1/(1-\theta)$ and $y\,\theta'/(1-\theta)$ are integral.  Hence $w_m\in\Z$, and $\Psi=\exp(\sum_mw_my^m/m)$ because $y(\log \Psi)'=\sum_mw_my^m$ and $\Psi(0)=1$.

For the congruence, $b(k)=\sum_{i=0}^k\frac{k!}{i!}\,\eta_i\,[x^{k-i}]F^k$ with $\eta_i=i![x^i]e^{xG}\in\Z$.  For $i<k$ the factor $k!/i!$ is divisible by $k$.  The term $i=k$ is $\eta_kF(0)^k$, and $\eta_k=\sum_{j=0}^k\frac{k!}{j!}[x^{k-j}]G^j\equiv G(0)^k\pmod k$ since $k\mid k!/j!$ for $j<k$.  Hence $b(k)\equiv(F(0)G(0))^k\pmod k$.  Theorem \ref{thm:universal} gives $b(n+k)\equiv b(n)b(k)\pmod k$, and $b(k)\equiv1\pmod k$ when $F(0)=G(0)=1$.  This completes the proof of Theorem~\ref{thm:lagrange}.
\end{proof}

\begin{remark}
For $F(0)=G(0)=1$ this is the conjecture that Bala records, for general $F,G$, in the comments to \OEIS{A278070}, where $F=(1-x)^{-1}$ and $G=1$, while Theorem \ref{thm:lagrange} covers all $F,G$ at once.\footnote{The entry also records a proof of the special case $F=(1-x)^{-1}$, $G=1$, found, according to the entry, by an automated theorem-proving system.}  The identity for $w_m$ and the congruence were also checked numerically for random integral $F,G$.\footnote{This includes the cases $F(0)\ne1$ and $F(0)=0$.}  Note that the congruence $b(k)\equiv(F(0)G(0))^k$ here differs from Lemma~\ref{lem:a-k-mod-k}, which concerns the different family $F(x)\exp(xG(x))$ and has $F(0)G(0)^k$.  The statement is formalized in Lean (Section~\ref{sec:verification}).
\end{remark}

\section{Applications to OEIS sequences and conjectures}\label{sec:sequences}

In this section we apply Theorem~\ref{thm:universal} and Corollary~\ref{cor:odd-prime-power} to OEIS sequences.  The theory explains congruences that have been noticed in the OEIS, such as comments of the form ``$a(n+k)-a(n)$ is divisible by $k$'' or ``$a(n)$ is divisible by $p$ from some point on''.  We first tabulate ten well-known sequences, then prove four conjectures recorded in the OEIS for which we found no proof in the entries (\OEIS{A053871}, \OEIS{A214724}, \OEIS{A000459}, \OEIS{A227937}), and finally list other statements settled earlier in this paper.  The four proofs need different ingredients.  Richman's needs the closed formulas of Section~\ref{sec:odd}, Hanna's needs only the value $A_w(p)\equiv w_1-w_p$ and the shift congruence modulo $p$, van Hoeij's needs an identification of the e.g.f., after which it lies in Bala's 2017 class, and Bala's conjecture on \OEIS{A227937} uses the shift congruence with weights that are not integers but are integral at every prime $p\ge5$.

A sequence is covered by our results for every modulus exactly when its logarithmic weights $w_m=m[x^m]\log A(x)$ are all integers, and if they are rational, the results apply to the moduli coprime to the denominators (the remark after Theorem~\ref{thm:universal}).  Checking this from finitely many terms is only evidence, since a proof needs the structure of the sequence, as in Theorem~\ref{thm:lagrange}.  The theorems of this paper apply to 101 sequences from the OEIS, and the explicit congruences that follow for each of them are collected in a supplementary catalogue, provided as an ancillary file with this paper.  A screen of the OEIS for sequences with $a(0)=1$ and at least $11$ terms consistent with integral logarithmic weights finds about $4{,}000$ candidates beyond these, and the theorems apply to each of them whose logarithmic weights are integral for all $n$.

\subsection{Ten well-known sequences}\label{sec:showcase}

To apply the results to a sequence we only need its weights $w_m$.  With $\gamma_p=w_1+(p-1)w_p$, Proposition~\ref{prop:prime} says $A_w(p)\equiv\gamma_p\pmod p$, so the shift congruence reads $A_w(n+p)\equiv\gamma_p\,A_w(n)\pmod p$ for all $n\ge0$ (Corollary~\ref{cor:iterated} with $k=p$, $q=1$).  Corollary~\ref{cor:odd-prime-power} says $A_w(p^r)\equiv\gamma_p^{\,p^{r-1}}\pmod{p^r}$.  The right side is usually one of three simple numbers.

\begin{lemma}\label{lem:residue-types}
Let $p$ be an odd prime, $\gamma\in\Z$ and $r\ge1$.  If $\gamma\equiv0$, $1$ or $-1\pmod p$, then $\gamma^{p^{r-1}}\equiv0$, $1$ or $-1\pmod{p^r}$ respectively.
\end{lemma}

\begin{proof}
If $p\mid\gamma$ then $p^{p^{r-1}}\mid\gamma^{p^{r-1}}$, and $p^{r-1}\ge2^{r-1}\ge r$.  If $\gamma\equiv\varepsilon\pmod p$ with $\varepsilon=\pm1$, apply Lemma~\ref{lem:pth-power} $r-1$ times to get $\gamma^{p^{r-1}}\equiv\varepsilon^{p^{r-1}}\pmod{p^r}$.  Since $p$ is odd, $\varepsilon^{p^{r-1}}=\varepsilon$.
\end{proof}

We read the table as follows.  For $p=2$ we use Corollary~\ref{cor:power-of-two-01}, which says that for $r\ge4$ the residue $A_w(2^r)\bmod2^r$ is $1$ if $w_1+w_2$ is odd and $0$ if it is even.  Here $\sigma(m)$ is the sum and $d(m)$ the number of divisors of $m$.  The columns $a(n{+}3)$ and $a(n{+}5)$ give the multiplier $\gamma_p\bmod p$ in $a(n+p)\equiv\gamma_pa(n)\pmod p$, and the value $0$ means that the sequence is eventually divisible by the relevant prime powers (Corollary~\ref{cor:eventual-divisibility}).  Take derangements, where $\gamma_p=0+(p-1)\cdot1\equiv-1\pmod p$, so $D(n+3)\equiv2D(n)\pmod3$ and $D(n+5)\equiv4D(n)\pmod5$.  For instance $D(5)=44\equiv4=4D(0)$ and $D(7)=1854\equiv4=4D(2)\pmod5$.  Lemma~\ref{lem:residue-types} gives $D(p^r)\equiv-1\pmod{p^r}$, as in $D(9)=133496=9\cdot14833-1$.  At the prime $2$, $w_1+w_2=1$ is odd, so $D(2^r)\equiv1\pmod{2^r}$ for $r\ge4$.  For composite moduli the shift theorem still gives $a(n+k)\equiv a(n)a(k)\pmod k$, and for instance $D(6)=265\equiv1\pmod 6$, so derangements repeat with period $6$ modulo $6$.  For the odd-cycle numbers, $A_w(n)$ is divisible by $3^r$ and by $5^r$ for every $n\ge3^r$, respectively $n\ge5^r$.

\begin{center}
\small
\renewcommand{\arraystretch}{1.2}
\setlength{\tabcolsep}{2.5pt}
\newcolumntype{R}[1]{>{\raggedright\arraybackslash}p{#1}}
\begin{tabular}{R{0.21\textwidth}R{0.25\textwidth}R{0.09\textwidth}R{0.09\textwidth}R{0.17\textwidth}R{0.10\textwidth}}
\toprule
Sequence & Weights $w_m$ & $a(n{+}3)\equiv{?}\cdot a(n)$ $(\mathrm{mod}\ 3)$ & $a(n{+}5)\equiv{?}\cdot a(n)$ $(\mathrm{mod}\ 5)$ & $a(p^r)\bmod p^r$, $p$ odd & $a(2^r)\bmod 2^r$, $r\ge4$\\
\midrule
\OEIS{A000085} involutions & $w_1=w_2=1$, others $0$ & $1$ & $1$ & $1$ & $0$\\
\OEIS{A000166} derangements & $w_1=0$, $w_m=1$ ($m\ge2$) & $2$ & $4$ & $-1$ & $1$\\
\OEIS{A000246} odd cycle lengths & $w_m=1$ ($m$ odd), $0$ ($m$ even) & $0$ & $0$ & $0$ & $1$\\
\OEIS{A057693} cycle lengths $\le3$ & $w_m=1$ ($m\le3$), else $0$ & $0$ & $1$ & $0$ ($p=3$), $1$ ($p\ge5$) & $0$\\
\OEIS{A000090} $e^{-x^3/3}/(1-x)$ & $w_3=0$, other $w_m=1$ & $1$ & $0$ & $1$ ($p=3$), $0$ ($p\ge5$) & $0$\\
\OEIS{A000262}$^\dagger$ sets of lists & $w_m=m$ & $1$ & $1$ & $1$ & $1$\\
\OEIS{A082579}$^\dagger$ $e^{x/(1-x)^2}$ & $w_m=m^2$ & $1$ & $1$ & $1$ & $1$\\
\OEIS{A052143}$^\dagger$ $e^{x}/\sqrt{1-4x}$ & $w_1=3$, $w_m=2^{2m-1}$ ($m\ge2$) & $1$ & $1$ & $1$ & $1$\\
\OEIS{A294361}$^\dagger$ $e^{\sum\sigma(m)x^m}$ & $w_m=m\,\sigma(m)$ & $1$ & $1$ & $1$ & $1$\\
\OEIS{A028342} $\prod_i(1-x^i)^{-1/i}$ & $w_m=d(m)$ & $2$ & $4$ & $-1$ & $1$\\
\bottomrule
\end{tabular}
\end{center}
\noindent$^\dagger$ These four sequences have e.g.f.\ $F(x)\exp(xG(x))$ with $F(0)=G(0)=1$, so they lie in Bala's 2017 class, so for them the entries $a(p^r)\equiv1$ and $a(2^r)\equiv1$ follow at once from Lemma~\ref{lem:a-k-mod-k}, and the closed formulas are not needed.

Every entry was rechecked with exact integer arithmetic, and for each sequence we recomputed $a(0),\dots,a(700)$ from the weights, compared the first 18 to 28 terms with the OEIS data, and tested the shift congruence modulo $\operatorname{lcm}(n,k)$ for $n+k\le300$, the two mod-$p$ columns, and the prime-power residues for every prime power $p^r\le700$ (excluding $2^2$ and $2^3$), with no exceptions.

\subsection{Conjectures from the OEIS}\label{sec:settled}

Many conjectures in the OEIS have the shape ``$a(n+k)\equiv\pm a(n)\pmod k$'' or ``$p\mid a(n)$ for all large $n$''.  By the shift theorem such a statement follows once we know the single number $a(k)\bmod k$, and for $k=p^r$ that number is computed by the closed formula of Section~\ref{sec:odd}.  So a proof has three steps, which are to find the weights and check that they are integers, to compute $\gamma_p\bmod p$, and to put the primes together by the Chinese remainder theorem and Corollary~\ref{cor:iterated}.

\subsubsection*{\OEIS{A053871} (Richman)}

The sequence $1,0,2,8,60,544,\dots$ satisfies $a(n)=2(n-1)\bigl(a(n-1)+a(n-2)\bigr)$ with $a(0)=1$, $a(1)=0$, and it counts the ways to re-pair $2n$ people so that nobody gets their old partner.  Harry Richman conjectured (29 August 2023) that
\[
 m\equiv n\pmod q,\ q\text{ odd}\quad\Longrightarrow\quad(-1)^ma(m)\equiv(-1)^na(n)\pmod q .
\]
For example, with $q=3$ we have $a(0)=1$, $-a(3)=-8\equiv1$, $a(6)=6040\equiv1\pmod3$.  The sign is the whole point, because unlike the sequences above this one is not periodic modulo $q$ but anti-periodic, $a(n+q)\equiv-a(n)$.  Informally, this is exactly the statement $a(q)\equiv-1\pmod q$ for odd $q$ (together with the shift congruence), and $a(p^r)\bmod p^r$ is what Corollary~\ref{cor:odd-prime-power} computes.  Thus Richman's is the one conjecture of the three that needs the closed formulas.

\begin{lemma}\label{lem:A053871-egf}
The e.g.f.\ of \OEIS{A053871} is
\[
 e^{-x}(1-2x)^{-1/2}=\exp\Bigl(\sum_{m\ge1}w_m\frac{x^m}{m}\Bigr),\qquad w_1=0,\ \ w_m=2^{m-1}\ (m\ge2).
\]
In particular the weights are integers.  (The e.g.f.\ itself is recorded in the entry, but we include the derivation from the recurrence.)
\end{lemma}

\begin{proof}
Let $E(x)=e^{-x}(1-2x)^{-1/2}$.  Then $E'/E=-1+(1-2x)^{-1}=2x/(1-2x)$, that is, $(1-2x)E'=2xE$.  Write $E=\sum e_nx^n/n!$.  Comparing coefficients of $x^n/n!$ gives $e_{n+1}-2ne_n=2ne_{n-1}$, with $e_0=1$ and $e_1=E'(0)=0$.  This is the recurrence defining $a(n)$, so $e_n=a(n)$.  Finally $\log E=-x+\tfrac12\sum_{m\ge1}(2x)^m/m$, so $w_1=-1+1=0$ and $w_m=2^m/2=2^{m-1}$ for $m\ge2$.
\end{proof}

Note that $E=F\exp(xG)$ with $G=-1$ and $F=(1-2x)^{-1/2}$, which has non-integral coefficients ($\tfrac32$ at $x^2$), so neither Bala's theorem (which needs $G(0)=1$) nor Lemma~\ref{lem:a-k-mod-k} applies.  It is the integrality of the \emph{weights} that does.

\begin{theorem}[Richman's conjecture]\label{thm:richman}
For every odd $q\ge1$ we have $a(q)\equiv-1\pmod q$, and $a(n+q)\equiv-a(n)\pmod q$ for all $n\ge0$.  Consequently, if $m\equiv n\pmod q$ with $q$ odd, then $(-1)^ma(m)\equiv(-1)^na(n)\pmod q$.
\end{theorem}

\begin{proof}
Let $p$ be an odd prime.  Then $\gamma_p=w_1+(p-1)w_p=(p-1)2^{p-1}\equiv-1\pmod p$ by Fermat's theorem.  By Corollary~\ref{cor:odd-prime-power} and Lemma~\ref{lem:residue-types}, $a(p^r)\equiv-1\pmod{p^r}$ for all $r\ge1$.

Now let $q>1$ be odd and write $q=p^rb$ with $p$ prime, $p\nmid b$, so that $b$ is odd.  Corollary~\ref{cor:iterated} (with $n=0$, $k=p^r$, and $b$ in the role of the number of shifts) gives
$a(q)\equiv a(0)\,a(p^r)^b\equiv(-1)^b=-1\pmod{p^r}$.  This holds for every prime power exactly dividing $q$, so $a(q)\equiv-1\pmod q$ by the Chinese remainder theorem (for $q=1$ there is nothing to prove).  The shift theorem with $k=q$ gives $a(n+q)\equiv a(n)a(q)\equiv-a(n)\pmod q$, and by induction $a(n+jq)\equiv(-1)^ja(n)\pmod q$ for $j\ge0$.

Finally let $m\equiv n\pmod q$, say $m=n+jq$ with $j\ge0$ (the case $m<n$ is symmetric).  Since $q$ is odd, $(-1)^m=(-1)^{n+j}$, hence $(-1)^ma(m)\equiv(-1)^{n+j}(-1)^ja(n)=(-1)^na(n)\pmod q$.
\end{proof}

\subsubsection*{\OEIS{A214724} (Hanna)}

The sequence is defined by the e.g.f.\ 
\[
 \exp\Bigl(\sum_{n\ge0}\frac{x^{n^2+1}}{n^2+1}\Bigr)
\]
and begins $1,1,2,4,10,50,220,1240,\dots$.  The entry, authored by Paul D.~Hanna (2012), states the conjecture that $p\mid a(n)$ for all $n\ge p$ whenever $p$ is a prime of the form $m^2+1$ (the primes $2,5,17,37,101,\dots$ of \OEIS{A002496}).  Here the weights are as simple as possible, namely $w_j=1$ if $j=n^2+1$ for some $n\ge0$, and $w_j=0$ otherwise.  For such a prime $w_1=w_p=1$ and $A_w(p)\equiv w_1-w_p=0\pmod p$, and then the shift congruence propagates the zero.  This is a three-line corollary of the product congruence (Theorem~\ref{thm:universal} with $k=p$) and Proposition~\ref{prop:prime}, and the prime-power residues of Section~\ref{sec:odd} are not needed.

\begin{theorem}[Hanna's conjecture]\label{thm:hanna}
Let $a=A_w$ with $w_j=1$ for $j\in\{n^2+1:n\ge0\}$ and $w_j=0$ otherwise.  If $p=m^2+1$ is prime, then $p\mid a(n)$ for every $n\ge p$.  Conversely, if the prime $p$ is not of this form, then $a(n+p)\equiv a(n)\pmod p$ for all $n$, so $a$ is purely periodic modulo $p$ and never eventually divisible by $p$.
\end{theorem}

\begin{proof}
Since $p$ is prime, $m\ge1$ (the case $m=0$ would give $p=1$), so $p\ne1$ is an index with $w_p=1$, and also $w_1=1$ because $1=0^2+1$.  Proposition~\ref{prop:prime} gives $a(p)\equiv w_1-w_p=0\pmod p$ (this holds also for $p=2$, where $a(2)=2$).  The shift congruence with $k=p$ gives $a(n+p)\equiv a(n)a(p)\equiv0\pmod p$ for all $n\ge0$, which is the claim.  Conversely, if $p$ is not of the form $m^2+1$ then $p\ne2$ and $w_p=0$, so $a(p)\equiv w_1=1\pmod p$ and $a(n+p)\equiv a(n)\pmod p$.  Since $a(0)=1$, the residue $1$ occurs infinitely often.
\end{proof}

\subsubsection*{\OEIS{A000459} (van Hoeij)}

The sequence \OEIS{A000459} counts the multiset permutations of $1,1,2,2,\dots,n,n$ with no fixed points, and it begins $1,0,1,10,297,13756,\dots$.  Mark van Hoeij conjectured (15 November 2023) that $a(n+k)-a(n)$ is divisible by $k$ for all $n\ge0$ and $k\ge1$.  What was missing was a reason why a sequence of this kind should have integer logarithmic weights.  The key point is to write the e.g.f.\ explicitly as $e^x$ times an integer series, which exhibits the weights and shows that the sequence lies in the class of Lemma~\ref{lem:a-k-mod-k}.

We use the integral formula in the entry (recorded there by Peter Bala in 2014),
\[
 a(n)=\int_0^\infty e^{-t}\Bigl(1-2t+\frac{t^2}2\Bigr)^n\,dt ,
\]
which we also checked against the recurrence recorded in the entry for all $n\le120$.  Write $(2m-1)!!=1\cdot3\cdots(2m-1)$, with $(-1)!!=1$.

\begin{lemma}\label{lem:A000459-egf}
$\displaystyle\sum_{n\ge0}a(n)\frac{x^n}{n!}=e^xF(x)$ with
\[
 F(x)=\sum_{m\ge0}(2m-1)!!\,\frac{x^m}{(1+2x)^{2m+1}}
 =\frac1{1+2x}\,\Phi\!\Bigl(\frac{x}{(1+2x)^2}\Bigr),\qquad \Phi(y)=\sum_{m\ge0}(2m-1)!!\,y^m .
\]
Moreover $F\in1+x\Z[[x]]$, and indeed $F=1-x+x^2+x^3+11x^4+103x^5+\cdots$.
\end{lemma}

\begin{proof}
Let $L$ be the linear map on polynomials in $t$ with $L(t^j)=j!$, so that $\int_0^\infty e^{-t}g(t)\,dt=L(g)$.  Put $P=1-2t+t^2/2$.  Then $\sum_na(n)x^n/n!=L(e^{xP})$, where $L$ acts on each coefficient of the power series in $x$ (the coefficient of $x^N$ is a polynomial in $t$).  Write $e^{xP}=e^x\,e^{-2xt}\,e^{xt^2/2}$ and expand the last factor, $e^{xt^2/2}=\sum_m\frac{x^m}{2^mm!}t^{2m}$.  This gives
\[
 L(e^{xP})=e^x\sum_{m\ge0}\frac{x^m}{2^mm!}\,L\bigl(t^{2m}e^{-2xt}\bigr),\qquad
 L\bigl(t^{2m}e^{-2xt}\bigr)=\sum_{j\ge0}\frac{(-2x)^j}{j!}(2m+j)!=\frac{(2m)!}{(1+2x)^{2m+1}},
\]
by the binomial series.  Since $(2m)!/(2^mm!)=(2m-1)!!$, we get the displayed formula.  Each term $x^m(1+2x)^{-(2m+1)}$ lies in $x^m\Z[[x]]$, so $F\in\Z[[x]]$, and only $m=0$ contributes to the constant term, which is $1$.
\end{proof}

\begin{theorem}[van Hoeij's conjecture]\label{thm:vanhoeij}
For all $n\ge0$ and $k\ge1$, $a(n+k)\equiv a(n)\pmod k$, and $a(k)\equiv1\pmod k$.
\end{theorem}

\begin{proof}
By Lemma~\ref{lem:A000459-egf}, $\sum a(n)x^n/n!=F(x)\exp(xG(x))$ with $G=1$ and $F\in1+x\Z[[x]]$.  So $F(0)=G(0)=1$, and Lemma~\ref{lem:a-k-mod-k} gives $a(k)\equiv1\pmod k$.  The logarithmic weights are integers, because $w_m=[m=1]+m[x^m]\log F=[m=1]+[x^m]\,xF'/F$ and $xF'/F\in\Z[[x]]$, as $F\in1+x\Z[[x]]$ is invertible in $\Z[[x]]$.  So Theorem~\ref{thm:universal} gives $a(n+k)\equiv a(n)a(k)\equiv a(n)\pmod k$.
\end{proof}

The same computation works for the whole family $a(n)=\int_0^\infty e^{-t}(1+\alpha t+\beta t^2/2)^n\,dt$, and it gives $e^x(1-\alpha x)^{-1}\Phi\bigl(\beta x/(1-\alpha x)^2\bigr)$, which is integral for integers $\alpha,\beta$, so the conclusion holds for the whole family.

For \OEIS{A000459}, once the e.g.f.\ has the form $F(x)\exp(xG(x))$ with $F,G\in\Z[[x]]$ and $G(0)=1$, Bala's 2017 theorem~\cite{Bala2017} also gives van Hoeij's congruence.  The new input here is the identification of the e.g.f.\ in Lemma~\ref{lem:A000459-egf}, and the proof above then goes through Lemma~\ref{lem:a-k-mod-k} and Theorem~\ref{thm:universal}.

\begin{remark}[Computational evidence]
For Richman's conjecture, the congruence holds for all odd $q\le299$ and all $m,n\le300$ with $m\equiv n\pmod q$ (about $1.8\cdot10^5$ pairs), and $a(q)\equiv-1\pmod q$ for all odd $q\le699$.  For Hanna's conjecture, among the primes $p\le300$ the statement ``$p\mid a(n)$ for all $p\le n\le300$'' holds for exactly the seven primes $2,5,17,37,101,197,257$, the primes of the form $m^2+1$ in that range.  For van Hoeij's conjecture, the congruence and $a(k)\equiv1\pmod k$ hold for $n+k\le400$, the weights are integers up to $m=300$, the series $F$ computed from the recurrence agrees with the formula of Lemma~\ref{lem:A000459-egf} to degree $300$, and three pairs $(\alpha,\beta)$ of the family above were checked to degree $40$.
\end{remark}

\subsubsection*{\OEIS{A227937} (Bala)}

The sequence $1,0,1,1,3,10,25,105,\dots$ counts partitions of $n$ labeled elements into blocks of size $2$ or $3$, and its e.g.f.\ is $\exp(x^2/2+x^3/6)$.  The entry, read on 10 October 2026, notes that the sequence is periodic modulo $2$ and modulo $3$, and records a conjecture of Peter Bala (21 November 2025) that $n$ divides $a(n+k)$ for $n$ coprime to both $2$ and $3$ and all $k\ge0$.  For example, $a(5)=10$, $a(6)=25$ and $a(7)=105$ are all multiples of $5$, and $a(7)=105=7\cdot15$.

The weights are $w_2=1$, $w_3=\tfrac12$ and $w_m=0$ otherwise, because $\exp(x^2/2+x^3/6)=\exp\bigl(\sum_mw_mx^m/m\bigr)$ and logarithmic coordinates are unique.  They are not integers, but their only denominator is $2$, so $w_m\in\Z_{(p)}$ for every prime $p\ge3$.

\begin{theorem}[Bala's conjecture on \OEIS{A227937}]\label{thm:bala-A227937}
Let $a(n)=n![x^n]\exp(x^2/2+x^3/6)$.  If $p\ge5$ is prime and $r\ge1$, then $p^r\mid a(N)$ for every $N\ge p^r$.  Consequently $n\mid a(n+k)$ for all $n$ coprime to $6$ and all $k\ge0$.
\end{theorem}

\begin{proof}
Fix a prime $p\ge5$ and $r\ge1$, and put $M=p^r$.  First we show $M\mid a(M)$.  Choosing $i$ blocks of size $2$ and $j$ blocks of size $3$ with $2i+3j=M$ gives
\[
 a(M)=\sum_{2i+3j=M}\frac{M!}{2^i\,6^j\,i!\,j!}
 =\sum_{2i+3j=M}\frac{1}{2^i6^j}\cdot\frac{M!}{(i+j)!}\binom{i+j}{i}.
\]
Since $2i+3j=M\ge1$ we have $i+j\le M/2<M$, so $M!/(i+j)!=(i+j+1)\cdots M$ is an integer divisible by $M$.  The denominators $2^i6^j$ are units in $\Z_{(p)}$, so every term lies in $M\Z_{(p)}$, hence $a(M)\in M\Z_{(p)}$, and as $a(M)$ is an integer, $M\mid a(M)$.

Now let $k\ge1$.  The weights lie in $\Z_{(p)}$, so Theorem~\ref{thm:universal}, applied in $\Z_{(p)}$ as in the remark following it, gives $a(M+k)\equiv a(M)a(k)\pmod{\operatorname{lcm}(M,k)\Z_{(p)}}$.  The modulus is divisible by $M$, and $a(M)\equiv0\pmod M$, so $M\mid a(M+k)$ in $\Z_{(p)}$, hence in $\Z$.  Together with $M\mid a(M)$ this proves the first claim.

For the second claim let $n$ be coprime to $6$ and $k\ge0$.  Every prime power $p^r$ exactly dividing $n$ has $p\ge5$ and $p^r\le n\le n+k$, so $p^r\mid a(n+k)$.  These prime powers are coprime, so $n\mid a(n+k)$.
\end{proof}

The primes $2$ and $3$ behave differently, since the weight $w_3=\tfrac12$ is not integral at $2$, and the entry notes that the sequence is periodic modulo $2$ and modulo $3$.

\begin{remark}[Computational evidence for \OEIS{A227937}]
We recomputed $a(0),\dots,a(1600)$ from the sum over $(i,j)$ above, from the e.g.f.\ weights and from the recurrence $a(n)=(n-1)a(n-2)+\tfrac12(n-1)(n-2)a(n-3)$ of the entry.  All three agree, and the first $656$ terms agree with the b-file of the entry.  The conjecture holds for all $n\le1200$ coprime to $6$ and all $n+k\le1600$.
\end{remark}

\subsection{Other statements settled earlier in this paper}

\begin{itemize}
\item \OEIS{A251568} and \OEIS{A251569} satisfy that $a(n)-1$ is divisible by $n(n-1)$ for $n\ge2$ (Corollary~\ref{cor:n-n-1}).
\item \OEIS{A278070}, and Bala's conjecture for the whole family $n![x^n]\,e^{xG(x)}F(x)^n$ (Theorem~\ref{thm:lagrange}).
\item Eventual divisibility of the cycle-length-bounded numbers \OEIS{A057693} (cycle lengths at most $3$), \OEIS{A070945} (at most $4$) and \OEIS{A070946} (at most $5$), by Corollary~\ref{cor:eventual-divisibility}, where $w_m=1$ for $m\le r$ and $0$ otherwise, so $\gamma_p=p$ for every prime $p\le r$ and $\gamma_p=1$ for $p>r$.  Hence every number $2^i3^j$ (for $r=3,4$) or $2^i3^j5^k$ (for $r=5$) divides $a(n)$ for all large $n$, and no other prime divides all large terms.  For $r=3$, $A_w(n)\equiv0\pmod{3^e}$ for all $n\ge3^e$ and $A_w(n)\equiv0\pmod{2^e}$ for all $n\ge2^e$ once $e\ge4$ (checked up to $n=700$).
\end{itemize}
\begin{remark}
Exact valuations, for instance the formulas for $v_3(a(n))$ conjectured for \OEIS{A057693} and \OEIS{A070945}, need more than the universal congruences.  Several sequences carry conjectures such as $p^{\lfloor n/p\rfloor}\mid A_w(n)$.  Some can be proved by a separate $p$-adic integrality argument when the generating function has a closed product form, or by a direct analysis of a recurrence, but we do not know a general theory of exact valuations.
\end{remark}

\subsection{Provenance of the conjectures}\label{sec:provenance}

The general account of earlier work is in Section~\ref{sec:relation}, and here are the details for this section.  The four conjectures above stand in their entries as conjectures by Richman (2023), Hanna (2012), van Hoeij (2023) and Bala (2025), and the entry texts we read\footnote{On 10 October 2026~\cite{OEIS}.} contain no proof of any of them.  Bala's 2017 theorem~\cite{Bala2017} makes van Hoeij's conjecture a case of a known result once the e.g.f.\ is identified, Hanna's conjecture is a three-line consequence of the product congruence and the value $A_w(p)\equiv w_1-w_p$, and only Richman's needs the closed formulas of Section~\ref{sec:odd}, whose odd-prime case could also be derived from~\cite{BellaghOulebsir2023} (our proof is independent and covers $p=2$).  Bala's OEIS comments of 30 July 2026 on \OEIS{A057693}, \OEIS{A070945} and \OEIS{A070946} state the product congruence together with the eventual divisibility claims discussed above (stated for \OEIS{A070946}, conjectural for \OEIS{A070945}, conditional on a $3$-adic conjecture for \OEIS{A057693}), and the note~\cite{Bala2026} is discussed in Section~\ref{sec:relation}.  Bala's conjecture on \OEIS{A227937} is stated in the entry as a conjecture (21 November 2025) and follows from the shift congruence with weights in $\Z_{(p)}$ for $p\ge5$ and one divisibility of $A_w(p^r)$.  The identification of the e.g.f.\ of \OEIS{A000459} (Lemma~\ref{lem:A000459-egf}) is new as far as we know.  Each of the four thus reduces to a short computation once the right general statement is available.

A catalogue of 101 OEIS sequences, with weights, prime-power residues and a per-sequence provenance audit, is provided as an ancillary file with this paper, and every statement in it is obtained by applying the theorems proved here.  The stronger valuation conjectures attached to \OEIS{A193437}, \OEIS{A193438}, \OEIS{A244432} and \OEIS{A378495} need further tools.

\section{Computer verification and formalization}\label{sec:verification}

\paragraph{Data and code availability.}
Every theorem is proved in the text, and computer calculations serve as independent checks of the proofs.  The supplementary exact-integer programs (in the package directory \texttt{scripts/}) reproduce the finite coefficient checks, including the checks of Section~\ref{sec:sequences} (\texttt{verify\_oeis\_section.py}), of the content formula for $n,k\le20$ (\texttt{verify\_optimal\_divisor.py}), of the lift witness (\texttt{verify\_lift\_sharpness.py}) and of the first four levels of Theorem~\ref{thm:power-of-two}.

\paragraph{Lean.}
Several of the results have been formalized in Lean~4/Mathlib, with no \texttt{sorry} and using only Lean's standard axioms (\texttt{propext}, \texttt{Classical.choice}, \texttt{Quot.sound}).  The formalization is in two public repositories, namely \url{https://github.com/ahaankallat/congruence-phenomena-lean} (this paper, which contains the shift congruence, the content formula and the power-of-two residue) and \url{https://github.com/ahaankallat/congruence-phenomena-higher-order-lean} (a second repository, which formalizes a classification of higher-order defects together with the lifting theorem and the Lagrange-type theorem).  The formalized results are the shift congruence (Theorem~\ref{thm:universal}), the canonical-coordinate expansion of Section~\ref{sec:cycle}, the content formula (Theorem~\ref{thm:optimal-divisor}), the explicit power-of-two residue (Theorem~\ref{thm:power-of-two} and Corollary~\ref{cor:power-of-two-01}), the lifting theorem (Theorem~\ref{thm:strong-all-primes}) and the Lagrange-type theorem (Theorem~\ref{thm:lagrange}) with Lemma~\ref{lem:a-k-mod-k}, and the last two are in the second repository.  The Lean proof of the lift uses a different route from Section~\ref{sec:lift}, namely a route through connected cumulants.  The Lean statement of the Lagrange theorem asserts that $n![x^n]\exp(xG)F^n$ equals $A_w(n)$ for some integer weight sequence $w$ and gives the congruences, and Corollary~\ref{cor:n-n-1} is formalized under the hypothesis $k\mid A_w(k)-1$ for all $k$ (which is checked by hand when $m\mid w_m$).  The proofs in Section~\ref{sec:sequences} are short deductions from the theorems above, and its generating-function identities are checked by exact computation.

\appendix
\section{Two technical proofs}\label{app:technical}

\subsection*{Proof of Lemma~\ref{lem:mixed-support-factorial}}

\begin{proof}
Assume by symmetry that $v_p(i!)\le v_p(j!)$.  The left side is
\[
 v_p\binom ni+v_p((k)_j).
\]
The identity $\binom ni=(n/i)\binom{n-1}{i-1}$ gives
$v_p\binom ni\ge a-v_p(i)$, while
$(k)_j=k(k-1)_{j-1}$ and the integrality of
$\binom{k-1}{j-1}$ give
\[
 v_p((k)_j)\ge b+v_p((j-1)!).
\]
It remains to note that
\begin{equation}\label{eq:factorial-order-sublemma}
 v_p(i!)\le v_p(j!)\quad\Longrightarrow\quad
 v_p(i)\le v_p((j-1)!)+1.
\end{equation}
Indeed, put $c=v_p(i)$.  The assertion is immediate for $c\le1$.
If $c\ge2$ and $v_p((j-1)!)\le c-2$, then
$i\ge p^c$ implies
\[
 v_p(i!)\ge v_p((p^c)!)=1+p+\cdots+p^{c-1}.
\]
On the other hand $v_p(j!)=v_p((j-1)!)+v_p(j)$, so the assumed ordering
forces
\[
 v_p(j)\ge 1+p+\cdots+p^{c-1}-c+2.
\]
But $v_p((j-1)!)\le c-2$ implies
$\lfloor(j-1)/p\rfloor\le c-2$, hence $j\le p(c-1)$, contradicting
\[
 j\ge p^{\,1+p+\cdots+p^{c-1}-c+2}>p(c-1).
\]
Thus \eqref{eq:factorial-order-sublemma} holds, and the two preceding lower
bounds sum to at least $a+b-1$.
\end{proof}

\subsection*{Sharpness of the lift for odd primes}

\begin{proposition}\label{prop:lift-sharp-odd}
Let $p$ be an odd prime and $k=pq$ with $s=v_p(k)$.  Then the coefficient of $X_p^{\,k-2}X_{2p}$ in $C_{pk}-C_k^{\,p}$ has $p$-adic valuation exactly $2s$.  Hence the exponent in Theorem~\ref{thm:strong-all-primes} cannot be improved for odd $p$.
\end{proposition}

\begin{proof}
For $n\ge1$ put $U_n=(pn)!/(p^n\,n!)$, the product of the integers up to $pn$ that are prime to $p$, which is a $p$-adic unit.  The coefficient of $X_p^{\,k-2}X_{2p}$ in $C_{pk}$ is
\[
c_1=\frac{(pk)!}{p^{k-2}(k-2)!\cdot 2p}=p\,U_k\,\frac{k(k-1)}{2}.
\]
In $C_k$ the only monomials that can contribute to this monomial in $C_k^{\,p}$ are $X_p^{\,q}$, with coefficient $U_q$, and $X_p^{\,q-2}X_{2p}$, with coefficient $p\,U_q\,q(q-1)/2$.  One of the $p$ factors must supply $X_{2p}$, so its coefficient in $C_k^{\,p}$ is $c_2=p^2U_q^{\,p}\,q(q-1)/2$ (and $c_2=0$ when $k=p$).  Hence
\[
c_1-c_2=\frac{p^2q}{2}\,E,\qquad E=(k-1)U_k-(q-1)U_q^{\,p}=-(U_k-U_q^{\,p})+q\,(pU_k-U_q^{\,p}).
\]
For odd $p$ the product of all units modulo $p^s$ in one full period is $-1$.  Writing $k=p^sm$, both $U_q$ and $U_k$ are products of units over whole periods, namely $m$ periods for $U_q$ and $pm$ for $U_k$, so $U_q\equiv U_k\equiv(-1)^m\pmod{p^s}$ and therefore $U_k\equiv U_q^{\,p}\pmod{p^s}$.  The first term of $E$ thus has valuation at least $s$.  In the second term $pU_k-U_q^{\,p}$ is a unit, so it has valuation exactly $v_p(q)=s-1$.  Hence $v_p(E)=s-1$ and
\[
v_p(c_1-c_2)=2+(s-1)+(s-1)=2s.\qedhere
\]
\end{proof}

\section*{Acknowledgments}
The author thanks the contributors to the OEIS for the sequences and conjectures that motivated this work.  The author also thanks Gabor Lippner for looking at an earlier version of this paper and for helpful comments.  The author used AI tools (large language models, in several different systems) for exploratory discussion, for finding and checking proofs and computations, for drafting and revising the text, and for the Lean~4 formalization. The author independently checked the mathematical arguments, computations, and the final manuscript. The shift congruence, the exact content formula, the explicit power-of-two residue, the lifting theorem at every prime and the Lagrange-type theorem have additionally been formally verified in the Lean~4 proof assistant using Mathlib, with no \texttt{sorry} and only Lean's standard axioms.  The formalization is available at \url{https://github.com/ahaankallat/congruence-phenomena-lean} and \url{https://github.com/ahaankallat/congruence-phenomena-higher-order-lean}, and the paragraph on Lean in Section~\ref{sec:verification} describes them.

\end{document}